\documentclass[11pt,english,a4paper]{smfart}
\usepackage[english]{babel}
\usepackage{amssymb,xspace}
\usepackage{amstext}
\usepackage{smfthm}
\theoremstyle{plain}
\usepackage{amsbsy,amssymb,amsfonts,latexsym}

\usepackage{enumerate}
\usepackage{graphics}

\title[Hilbertian Jamison sequences and rigid dynamical systems]{Hilbertian Jamison sequences and rigid dynamical systems}

\author{Tanja Eisner}
\address{Korteweg de Vries Institut voor Wiskunde,
Universiteit van Amsterdam, P.O. Box 94248, 1090 GE Amsterdam, The Netherlands}
\email{T.Eisner@uva.nl}

\author{Sophie Grivaux}
\address{
Laboratoire Paul Painlev\' e, UMR 8524, Universit\'e  Lille 1, Cit\' e Scientifique, 59655 Villeneuve d'Ascq
Cedex, France}
\email{grivaux@math.univ-lille1.fr}



\thanks{This work was partially supported by ANR-Projet Blanc DYNOP,
the European Social Fund and the Ministry of Science, Research
and the Arts Baden-W\"urttemberg.}

\def\T{\ensuremath{\mathbb T}}
\def\R{\ensuremath{\mathbb R}}
\def\Z{\ensuremath{\mathbb Z}}

\def\Q{\ensuremath{\mathbb Q}}
\def\N{\ensuremath{\mathbb N}}
\def\P{\ensuremath{\mathbb P}}

\newcommand{\sep}{separable}
\newcommand{\js}{Jamison sequence}
\newcommand{\hjs}{Hilbertian Jamison sequence}

\newcommand{\fhy}{frequently hypercyclic}
\newcommand{\ops}{operators}
\newcommand{\op}{operator}

\newcommand{\erg}{ergodic}
\newcommand{\eve}{eigenvector}
\newcommand{\eva}{eigenvalue}
\newcommand{\ps}{perfectly spanning set of eigenvectors associated to unimodular
eigenvalues}
\newcommand{\wrt}{with respect to}
\newcommand{\bs}{backward shift}
\newcommand{\proba}{probability}
\newcommand{\ga}{Gaussian}

\newcommand{\inv}{invariant}

\newcommand{\mea}{measure}
\newcommand{\nd}{non-degenerate}
\newcommand{\mpt}{measure-preserving transformation}

\newcommand{\wmx}{weakly mixing}

\newcommand{\rg}{rigid}
\newcommand{\ur}{uniformly rigid}
\newcommand{\rs}{rigidity sequence}
\newcommand{\urs}{uniform rigidity sequence}

\newcommand{\ifff}{if and only if}

\newcommand{\pss}[2]{\ensuremath{{\langle #1,#2\rangle}}}

\newcommand{\set}[1]{\left\{#1\right\}}

\newcommand{\eps}{\varepsilon}

\newtheorem{theorem}{Theorem}[section]

\newtheorem{lemma}[theorem]{Lemma}

\newtheorem{proposition}[theorem]{Proposition}

\newtheorem{corollary}[theorem]{Corollary}

{\theoremstyle{definition}}

{\theoremstyle{definition}\newtheorem{example}[theorem]{Example}}

\newtheorem{fact}[theorem]{Fact}

{\theoremstyle{definition}\newtheorem{definition}[theorem]{Definition}}

{\theoremstyle{definition}}

{\theoremstyle{definition}\newtheorem{remark}[theorem]{Remark}}

{\theoremstyle{definition}\newtheorem*{FFC Criterion}{Frequent
Faber-hypercyclicity Criterion}}
\newtheorem*{Hypercyclicity Criterion}{Hypercyclicity Criterion}
{\theoremstyle{definition}\newtheorem*{GS Criterion}{Godefroy-Shapiro
Criterion}}
\def\piednote#1{\let\oldfn=\thefootnote\def\thefootnote{}\footnote{\noindent#1}%
\addtocounter{footnote}{-1}\def\thefootnote{\oldfn}}

\begin{document}

\begin{abstract}
A strictly increasing sequence $(n_{k})_{k\geq 0}$ of positive integers is said to be a \emph{Hilbertian Jamison sequence} if for any bounded
\op\ $T$ on a \sep\ Hilbert space
such that $\sup_{k\geq 0}||T^{n_{k}}||<+\infty $, the set of eigenvalues of modulus $1$ of $T$ is at most countable. We first give a complete characterization of such sequences. We then turn to the study of rigidity sequences $(n_{k})_{k\geq 0}$ for weakly mixing dynamical systems on measure spaces, and give various conditions, some of which are closely related to the Jamison condition, for a sequence to be a rigidity sequence. We obtain on our way a complete characterization of topological rigidity and uniform rigidity sequences for linear dynamical systems, and we construct in this framework examples of dynamical systems which are both weakly mixing in the measure-theoretic sense and uniformly rigid.
\end{abstract}

\maketitle

\section{Introduction and main results}
We are concerned in this paper with the study of certain dynamical systems, in particular linear dynamical systems. Our main aim is the study of \emph{rigidity sequences} $(n_{k})_{k\geq 0}$ for weakly mixing dynamical systems on measure spaces, and we present tractable conditions on the sequence $(n_{k})_{k\geq 0}$ which imply that it is (or not) a rigidity sequence. Our conditions on the sequence $(n_{k})_{k\geq 0}$ come in part from the study of the so-called \emph{Jamison sequences}, which appear in the description of the relationship between partial power-boundedness of an \op\ on a \sep\ Banach space and the size of its unimodular point spectrum.
\par\smallskip
Let us now describe our results more precisely.
\subsection{A characterization of Hilbertian Jamison sequences}
Let $X$ be a \sep\ infinite-dimensional complex Banach space, and let $T\in\mathcal{B}(X)$ be a bounded operator on $X$. We are first going to study here the relationship between the behavior of the sequence $||T^n||$ of the norms of the powers of $T$, and the size of the unimodular point spectrum $\sigma_p(T)\cap\T$, i.e. the set of \eva s of $T$ of modulus $1$. It is known since an old result of Jamison \cite{J} that a slow growth of $||T^n||$ makes $\sigma_p(T)\cap\T$ small, and vice-versa. More precisely, the result of \cite{J} states that if $T$ is power-bounded, i.e. $\sup_{n\geq 0}||T^n||<+\infty$, then $\sigma_p(T)\cap\T$ is at most countable. For a sample of the kind of results which can be obtained in the other direction, let us mention the following result of Nikolskii \cite{Ni}: if $T$ is a bounded \op\ on a \sep\ Hilbert space such that $\sigma_p(T)\cap\T$ has positive Lebesgue \mea, then the series $\sum_{n\geq 0}||T^n||^{-2}$ is convergent. This has been generalized by Ransford in the paper \cite{R}, which renewed the interest in these matters. In particular Ransford started to investigate in \cite{R} the influence of partial power-boundedness of an \op\ on the size of its unimodular point spectrum. Let us recall the following definition:

\begin{definition}
Let $(n_{k})_{k\geq 0}$ be an increasing sequence of
positive integers,
and $T$ a bounded linear operator on the space $X$. We say that $T$ is \emph{partially
 power-bounded \wrt}\ $(n_{k})$ if $\sup_{k\geq
 0}||T^{n_{k}}||<+\infty$.
\end{definition}

In view of the result of Jamison, it was natural to investigate whether the partial power-boundedness of $T$ \wrt\ $(n_{k})$ implies that $\sigma_p(T)\cap\T$ is at most countable. It was shown in \cite{RR} by Ransford and Roginskaya that it is not the case: if $n_k=2^{2^k}$ for instance, there exist a \sep\ Banach space $X$ and $T\in\mathcal{B}(X)$ such that $\sup_{k\geq
 0}||T^{n_{k}}||$ is finite while $\sigma_p(T)\cap\T$ is uncountable. This question was investigated further in \cite{BaG1} and \cite{BaG2}, where the following definition was introduced:

\begin{definition}
Let $(n_{k})_{k\geq 0}$ be an increasing sequence of
integers. We say that $(n_{k})_{k\geq 0}$ is a \emph{Jamison sequence}
if for any separable Banach space $X$ and any bounded
\op\ $T$ on $X$, $\sigma _{p}(T)\cap\T$ is at most countable
as soon as $T$ is partially power-bounded \wrt\ $(n_{k})$.
\end{definition}

Whether $(n_{k})_{k\geq 0}$ is a Jamison sequence or not depends of
course on features of the sequence such as its growth, its
arithmetical properties, etc. A complete characterization of Jamison
sequences was obtained in \cite{BaG2}. It
is formulated
as follows:

\begin{theorem}\label{th0}
Let $(n_{k})_{k\geq 0}$ be an increasing sequence of
integers with $n_{0}=1$. The following assertions are equivalent:
\begin{itemize}
\item[(1)] $(n_{k})_{k\geq 0}$ is a \js;

\item[(2)] there exists a positive real number $\varepsilon $ such that
for every $\lambda \in\T\setminus \set{1}$, $$\sup_{k\geq 0}|\lambda ^{n_{k}}-1|
\geq \varepsilon .$$
\end{itemize}
\end{theorem}

Many examples of Jamison and non-Jamison sequences were obtained in \cite{BaG1} and \cite{BaG2}. Among the examples of non-Jamison sequences, let us mention the sequences $(n_{k})_{k\geq 0}$ such that $n_{k+1}/n_{k}$ tends to infinity, or such that $n_{k}$ divides $n_{k+1}$ for each $k\geq 0$ and $\limsup n_{k+1}/n_{k}=+\infty $. Saying that $(n_{k})_{k\geq 0}$ is not a Jamison sequence means that there exists a \sep\ Banach space $X$ and $T\in\mathcal{B}(X)$ such that $\sup_{k\geq
 0}||T^{n_{k}}||<+\infty$ and $\sigma _{p}(T)\cap\T$ is uncountable. But the space $X$ may well be extremely complicated: in the proof of Theorem \ref{th0}, the space is obtained by a rather involved renorming of a classical space such as $\ell_2$ for instance. This is a drawback in applications, and this is why it was investigated in \cite{BaG1} under which conditions on the sequence $(n_{k})_{k\geq 0}$ it was possible to construct partially power-bounded operators \wrt\  $(n_{k})_{k\geq 0}$ with uncountable unimodular point spectrum on a Hilbert space. It was proved in \cite{BaG1} that if the series $\sum_{k\geq 0}(n_{k}/n_{k+1})^2$ is convergent, there exists a bounded \op\ $T$ on a \sep\ Hilbert space $H$ such that $\sup_{k\geq
 0}||T^{n_{k}}||<+\infty$ and $\sigma _{p}(T)\cap\T$ is uncountable. But this left open the characterization of Hilbertian Jamison sequences.

\begin{definition}
We say that $(n_{k})_{k\geq 1}$ is a \emph{Hilbertian Jamison sequence}
if for any bounded
\op\ $T$ on
a \sep\ infinite-dimensional complex Hilbert space
which is partially power-bounded \wrt\ $(n_{k})$, $\sigma _{p}(T)\cap\T$ is at most countable.
\end{definition}

Obviously a Jamison sequence is a Hilbertian Jamison sequence.
Our first goal in this paper is to prove the somewhat surprising fact that the converse is true:

\begin{theorem}\label{th1}
Let $(n_{k})_{k\geq 0}$ be an increasing sequence of
integers. Then $(n_{k})_{k\geq 0}$ is a \hjs\ \ifff\
it is a \js.
\end{theorem}

Contrary to the proofs of \cite{BaG1} and \cite{BaG2}, the proof of Theorem \ref{th1} is completely explicit: the operators with $\sup_{k\geq 0}||T^{n_{k}}||<+\infty $  and $\sigma  _{p}(T)\cap\T$ uncountable which we construct are perturbations by a weighted \bs\ on $\ell_2$ of a diagonal \op\ with unimodular diagonal coefficients. The construction here bears some similarities with a construction carried out in a different context in \cite{DFGP} in order to obtain \fhy\ \ops\ on certain Banach spaces.

\subsection{Ergodic theory and rigidity sequences}
Let $(X,\mathcal{F},\mu)$ be a finite \mea\ space where $\mu $ is a positive regular finite Borel \mea, and let $\varphi $ be a measurable transformation of $(X,\mathcal{F},\mu)$. We recall here that $\varphi$ is said to \emph{preserve the measure} $\mu $ if $\mu (\varphi^{-1}(A))=\mu (A)$ for any $A\in\mathcal{F}$, and that $\varphi$ is said to be \emph{ergodic} \wrt\ $\mu $ if for any $A,B \in\mathcal{F}$ with $\mu (A)>0$ and $\mu (B)>0$, there exists an $n\geq 0$ such that $\mu (\varphi^{-n}(A)\cap B)>0$, where $\varphi^{n}$ denotes the $n^{th}$ iterate of $\varphi$. Equivalently, $\varphi$ is \erg\ \wrt\ $\mu $ \ifff\
$$\frac{1}{N}\sum_{n=1}^{N}\mu (\varphi^{-n}(A)\cap B)\to \mu(A)\mu(B)\quad \textrm{ as } N\to +\infty \quad \textrm{for every } A,B\in \mathcal{F}.$$
This leads to the notion of \wmx\ \mea-preserving transformation of $(X,\mathcal{F},\mu)$: $\varphi$ is \emph{\wmx}\ if
$$\frac{1}{N}\sum_{n=1}^{N}|\mu (\varphi^{-n}(A)\cap
B)-\mu(A)\mu(B)|\to 0 \quad \textrm{ as } N\to +\infty
\quad \textrm{for every } A,B\in \mathcal{F}.$$ It is well-know that $\varphi$ is \wmx\ \ifff\ $\varphi\times
\varphi$ is an \erg\ transformation of $X\times X$ endowed with the
product \mea\ $\mu\times\mu$. We refer the reader to \cite{CFS},
\cite{Pet} or
\cite{Wa} for instance for more about \erg\ theory of dynamical systems and various examples.
\par\smallskip
Our interest in this paper lies in \wmx\ \emph{rigid} dynamical systems:
\begin{definition}
A \mpt\ of  $(X,\mathcal{F},\mu)$ is said to be \emph{rigid} if there exists a sequence $(n_{k})_{k\geq 0}$ of integers such that for any $A\in \mathcal{F}$, $\mu(\varphi^{-n_{k}}(A)\bigtriangleup A)\to 0$ as $k\to +\infty $.
\end{definition}
If $U_{\varphi}$ denotes the isometry on $L^{2}(X,\mathcal{F},\mu)$ defined by $U_{\varphi}f:=f\circ\varphi$ for any $f\in L^{2}(X,\mathcal{F},\mu)$, it is not difficult to see that $\varphi$ is rigid \wrt\ the sequence $(n_{k})_{k\geq 0}$ \ifff\
 $||U_{\varphi}^{n_{k}}f-f||\to 0$ as $k\to +\infty $ for any $f\in L^{2}(X,\mathcal{F},\mu)$. The function $f$ itself is said to be rigid \wrt\ $(n_{k})_{k\geq 0}$ if $||U_{\varphi}^{n_{k}}f-f||\to 0$. Rigid functions play
a major role in the study of mildly mixing dynamical systems, as introduced by Furstenberg and Weiss in \cite{FW}, and rigid weakly mixing systems are intensively studied, see for instance the works
 \cite{1}, \cite{2}, \cite{3} or \cite{4} as well as the references therein for some examples of results and methods. Let us just mention here the fact that \wmx\ rigid transformations
 of $(X,\mathcal{F},\mu)$ form a residual subset of the set of all \mpt s of $(X,\mathcal{F},\mu)$ for
the weak topology \cite{K}. A \emph{rigidity sequence} is defined as follows:

\begin{definition}
Let $(n_{k})_{k\geq 0}$  be a strictly increasing sequence of
positive integers.
We say that $(n_{k})_{k\geq 0}$ is a \emph{rigidity sequence} if there exists a \mea\ space $(X,\mathcal{F},\mu)$ and a \mpt\ $\varphi$ of $(X,\mathcal{F},\mu)$ which is \wmx\ and rigid \wrt\ $(n_{k})_{k\geq 0}$.
\end{definition}
\begin{remark}\label{rem:inv}
In the literature one often defines rigidity sequences as sequences
for which there exists an \emph{invertible} \mpt\ which is \wmx\ and
rigid \wrt\ $(n_{k})_{k\geq 0}$. In fact, these two definitions are
equivalent since every rigid \mpt\ $\varphi$ is invertible (in the measure-theoretic sense).
An easy way to see it is to consider the induced
isometry $U_\varphi$ defined above. Since $\varphi$ is invertible if
and only if $U_\varphi$ is so, it suffices to show that $U_\varphi$ is
invertible. By the decomposition theorem for contractions due to
Sz.-Nagy, Foia{\c{s}} \cite{NaFo}, $U_\varphi$ can be decomposed into a direct sum of a
unitary operator and a weakly stable operator. Since
$\lim_{k\to\infty}U_\varphi^{n_k}=I$ in the weak operator topology
(see Fact \ref{fact2} below), the weakly
stable part cannot be present and thus $U_\varphi$ is a unitary operator and
$\varphi$ is invertible.
\end{remark}

Rigidity sequences are in a sense already characterized: $(n_{k})_{k\geq 0}$ is a \rs\ \ifff\ there exists a continuous \proba\ \mea\ $\sigma  $ on the unit circle $\T$ such that
$$\int_{\T}|\lambda ^{n_{k}}-1|d\sigma  (\lambda )\to 0\quad \textrm{ as } k\to +\infty $$
(see Section 3.1 for more details). Still, there is a lack of practical conditions which would enable us to check easily whether a given sequence $(n_{k})_{k\geq 0}$ is a \rs. It is the second aim of this paper to provide such conditions. Some of them can be initially found in the papers \cite{BaG1} and \cite{BaG2} which study Jamison sequences in the Banach space setting, and they turn out to be relevant for the study of rigidity. We show for instance that if ${n_{k+1}}/{n_{k}}$ tends to infinity as $k$ tends to infinity, $(n_{k})_{k\geq 0}$ is a \rs\ (see Example \ref{ex1} and Proposition \ref{prop3}). If $(n_{k})_{k\geq 0}$ is any sequence such that $n_{k}$ divides $n_{k+1}$ for any $k\geq 0$, $(n_{k})_{k\geq 0}$ is a \rs\ (Propositions \ref{prop22} and \ref{prop33}). We also give some examples involving the
denominators of the partial quotients in the continuous fraction expansion of some irrational numbers (Examples \ref{ex66} and \ref{ex2}), as well as an example of a \rs\ such that $n_{k+1}/n_{k}\to 1$ (Example \ref{ex77}). In the other direction, it is not difficult to show that if $n_{k}=p(k)$ for some polynomial $p\in\Z[X]$ with $p(k)\geq 0$ for any $k$, $(n_{k})_{k\geq 0}$ cannot be a \rs\ (Example \ref{ex44}), or that the sequence of prime numbers cannot be a \rs\ (Example \ref{ex55}). Other examples of non-\rs s can be given (Example \ref{ex5}) when the sequences $(n_{k}x)_{k\geq 0}$, $x\in [0,1]$, have suitable equirepartition properties.

\subsection{Ergodic theory and rigidity for linear dynamical systems}
If $T$ is a bounded \op\ on a \sep\ Banach space $X$, it is sometimes possible
to endow the space $X$ with a suitable \proba\ \mea\ $m$, and to consider
$(X,\mathcal{B},m,T) $ as a measurable dynamical system. This was
first done in the seminal work \cite{Fl} of Flytzanis, and the study
was continued in the papers \cite{BG2} and \cite{BG3}. If $X$ is a
\sep\ complex Hilbert space which we denote by $H$, $T\in \mathcal{B}(H)$ admits a \nd\ \inv\ \ga\ \mea\ \ifff\ its \eve s associated to \eva s of modulus $1$ span a dense subspace of $H$, and it is \erg\ (or here, equivalently, \wmx)\ \wrt\ such a \mea\ \ifff\ it has a \ps\ (see Section 2.1 for the definitions) -- this condition very roughly means that $T$ has ``plenty'' of such \eve s, ``plenty'' being quantified by a continuous \proba\ \mea\ on the unit circle.
\par\smallskip
It comes as a natural question to describe \rs s in the framework of  linear dynamics, and it is not difficult to show that if $(n_{k})_{k\geq 0}$ is a \rs, there exists a bounded \op\ on $H$ which is \wmx\ and \rg\ \wrt\ $(n_{k})_{k\geq 0}$ (see Section 4.1).
Thus, every rigidity sequence can be realized in a linear Hilbertian
measure-preserving dynamical system.
However, the answer is not so simple when one considers topological and uniform rigidity, which are topological analogues of the (measurable) notion of rigidity. These notions were introduced by Glasner and Maon in the paper \cite{GM} for continuous dynamical systems on compact spaces:

\begin{definition}
 Let $(X,d)$ be a compact metric space, and let $\varphi$ be a continuous self-map of $X$. We say that $\varphi$ is \emph{topologically rigid} \wrt\ the sequence $(n_{k})_{k\geq 0}$ if $\varphi^{n_{k}}(x)\to x$ as $k\to +\infty $ for any $x\in X$, and that $\varphi$ is \emph{uniformly rigid} \wrt\ $(n_{k})_{k\geq 0}$ if $$\sup_{x\in X}d(\varphi^{n_{k}}(x),x)\to 0 \quad \textrm{ as } k\to +\infty .$$
\end{definition}

It is easy to check, using the Lebesgue dominated convergence theorem, that a topologically or uniformly rigid dynamical system is rigid.
Uniform rigidity is studied in \cite{GM}, where in particular uniformly \rg\ and topologically \wmx\ transformations are constructed, see also \cite{5}, \cite{6} and \cite{JKLSS} for instance.
Recall that $\varphi$ is said to be \emph{topologically \wmx}\ if for any non-empty open subsets $U_1,U_2,V_1,V_2$ of $X$, there exists an integer $n$ such that $\varphi^{-n}(U_1)\cap V_1$ and $\varphi^{-n}(U_2)\cap V_2$ are both non-empty (topological weak mixing is the topological analogue of the notion of measurable weak mixing). Uniform rigidity sequences are defined in \cite{JKLSS}:

\begin{definition}
Let $(n_{k})_{k\geq 0}$  be a strictly increasing sequence of integers. We say that $(n_{k})_{k\geq 0}$ is a \emph{uniform rigidity sequence} if there exists a compact dynamical system $(X,d,\varphi)$ with $\varphi$ a continuous self-map of $X$, which is topologically \wmx\ and uniformly rigid \wrt\ $(n_{k})_{k\geq 0}$.
\end{definition}

The question of characterizing uniform rigidity sequences
is still open, as well as the question  \cite{JKLSS}
whether there exists  a compact dynamical system $(X,d,\varphi)$ with $\varphi$ continuous, which would be both \wmx\ \wrt\ a certain $\varphi$-\inv\ \mea\ $\mu$ on $X$ and uniformly rigid.
\par\smallskip
We investigate these two questions in the framework of linear dynamics. Of course we have to adapt the definition of uniform rigidity to this setting, as a Banach space is never compact.

\begin{definition}
 Let $X$ be complex \sep\ Banach space, and let $\varphi$ be a continuous transformation of $X$. We say that $\varphi$ is \emph{uniformly rigid} \wrt\ $(n_{k})_{k\geq 0}$ if for any bounded subset $A$ of $X$, $$\sup_{x\in A}||\varphi^{n_{k}}(x)-x||\to 0 \quad \textrm{ as } k\to +\infty .$$
\end{definition}

When $T$ is a bounded linear \op\ on $X$, $T$ is \ur\ \wrt\ $(n_{k})_{k\geq 0}$ \ifff\ $||T^{n_{k}}-I||\to 0$ as $k\to +\infty $. We prove the following theorems:

\begin{theorem}\label{th2}
Let $(n_{k})_{k\geq 0}$  be an increasing sequence of integers with $n_{0}=1$. The following assertions are equivalent:
\begin{itemize}
 \item[(1)]  for any $\varepsilon >0$ there exists a $\lambda \in\T\setminus\{1\}$ such that $$
\sup_{k\geq 0}|\lambda ^{n_{k}}-1|\leq\varepsilon \quad \textrm{ and }\quad |\lambda ^{n_{k}}-1|\to 0
\quad \textrm{ as }\quad k\to +\infty ;$$

\item[(2)] there exists a bounded linear \op\ $T$ on a \sep\
  Banach space $X$ such that $\sigma_p(T)\cap \T$ is uncountable and
  $T^{n_k}x\to x$ as $k\to\infty$ for every $x\in X$;

\item[(3)]
there exists a bounded linear \op\ $T$ on a \sep\ Hilbert space $H$ such that $T$ admits a \nd\ \inv\ \ga\ \mea\ \wrt\ which $T$ is \wmx\ and $T^{n_{k}}x \to x$ as $k \to +\infty $ for every $x\in H$, i.e. $T$ is topologically \rg\ \wrt\ $(n_{k})_{k\geq 0}$.
\end{itemize}
\end{theorem}

We also have a characterization for uniform rigidity in the linear setting:
\begin{theorem}\label{th2bis}
Let $(n_{k})_{k\geq 0}$  be an increasing sequence of integers. The following assertions are equivalent:
\begin{itemize}
 \item[(1)]  there exists an uncountable subset $K $ of $\T$ such that $\lambda ^{n_{k}}$ tends to $1$ uniformly on $K$;

\item[(2)] there exists a bounded linear \op\ $T$ on a \sep\
  Banach space $X$ such that $\sigma_p(T)\cap \T$ is uncountable and
  $||T^{n_k}-I||\to 0$ as $k\to\infty$;

\item[(3)] there exists a bounded linear \op\ $T$ on a \sep\ Hilbert space $H$ such that $T$ admits a \nd\ \inv\ \ga\ \mea\ \wrt\ which $T$ is \wmx\ and $||T^{n_k}-I||\to 0$ as $k\to\infty$, i.e. $T$ is \ur\ \wrt\ $(n_{k})_{k\geq 0}$.
\end{itemize}
\end{theorem}

In particular we get a positive answer to a question of \cite{JKLSS} in the context of linear dynamics:

\begin{corollary}\label{cor}
 Any sequence $(n_{k})_{k\geq 0}$ such that $n_{k+1}/n_{k}$ tends to infinity, or such that $n_{k}$ divides $n_{k+1}$ for each $k$ and $\limsup
n_{k+1}/n_{k}=+\infty $ is a uniform rigidity sequence for linear dynamical systems, and measure-theoretically \wmx\ \ur\ systems do exist in this setting.
\end{corollary}

\par\bigskip

After this paper was submitted for publication,  V. Bergelson, A. Del Junco, M. Lema\'nczyk and J. Rosenblatt sent us a preprint ``Rigidity and non-recurrence along sequences'' \cite{preprint}, in which they independently investigated for which sequences there exists a \wmx\ transformation which is \rg\ \wrt\ this sequence.
A substantial part of the results of Section 3 of the present paper is also obtained in \cite{preprint}, often with different methods. We are very grateful to V. Bergelson, A. Del Junco, M. Lema\'nczyk and J. Rosenblatt for making their preprint available to us.

\section{Hilbertian Jamison sequences}
Our aim in this section is to prove Theorem \ref{th1}. Clearly, if $(n_{k})_{k\geq 0}$ is a \js, it is automatically a \hjs, and the difficulty lies in the converse direction: using Theorem \ref{th0}, we start from a sequence $(n_{k})_{k\geq 0}$ such that for any $\varepsilon >0$ there is a $\lambda \in\T\setminus\{1\}$ such that $\sup_{k\geq 0}|\lambda ^{n_{k}}-1|\leq \varepsilon $, and we have to construct out of this  a bounded \op\ on a \emph{Hilbert space} which is partially power-bounded \wrt\ $(n_{k})_{k\geq 0}$ and which has uncountably many \eva s on the unit circle. We are going to prove a stronger theorem, giving a more precise description of the \eve s of the \op:

\begin{theorem}\label{th1bis}
Let $(n_{k})_{k\geq 0}$ be an increasing sequence of
integers with $n_{0}=1$ such that for any $\varepsilon>0$ there exists a $\lambda \in\T\setminus \set{1}$ such that $$\sup_{k\geq 0}|\lambda ^{n_{k}}-1|
\leq \varepsilon .$$ Let $\delta>0$ be any positive number. There exists a bounded linear \op\ $T$ on the Hilbert space $\ell_2(\N)$  such that $T$ has perfectly spanning unimodular \eve s and $$\sup_{k\geq 0}||T^{n_{k}}||\leq 1+\delta.$$
In particular the unimodular point spectrum of $T$ is uncountable.
\end{theorem}

Before embarking on the proof, we need to define precisely the notion of perfectly spanning unimodular \eve s and explain its relevance here.

\subsection{A criterion for ergodicity of linear dynamical systems}
Let $H$ be a complex \sep\ infinite-dimensional Hilbert space.

\begin{definition}
We say that a bounded linear operator $T$ on $H$ has a \emph{\ps}\ if
there exists a continuous probability \mea\ $\sigma  $ on the unit
circle $\T$ such that for any
Borel
subset $B$ of $\T$ with $\sigma  (B)=1$, we have $\overline{\textrm{sp}}[\ker(T-\lambda I) \textrm{ ; }\lambda \in B]=H$.
\end{definition}

When $T\in \mathcal{B}(H)$ has a \ps, there exists a \ga\ \proba\ \mea\ $m$ on $H$ such that:

-- $m$ is $T$-\inv;

-- $m$ is \nd, i.e. $m(U)>0$ for any non-empty open subset $U$ of $H$;

-- $T$ is \wmx\ \wrt\ $m$.

See \cite{BG3} for extensions to the Banach space setting, and the book \cite[Ch. 5]{BM}. In the Hilbert space case, the converse of the assertion above is also true: if $T\in \mathcal{B}(H)$ defines a \wmx\ \mpt\ \wrt\ a \nd\ \ga\ \mea, $T$ has perfectly spanning unimodular \eve s.
\par\smallskip
A way to check this spanning property of the \eve s is to
use the following criterion, which was proved in \cite[Th. 4.2]{G}:

\begin{theorem}\label{th0bis}
Let $X$ be a complex \sep\ infinite-dimensional Banach space, and let $T$ be a bounded \op\ on $X$. Suppose that there exists a sequence $(u_{i})_{i\geq 1}$ of vectors of $X$ having the following properties:
 \begin{itemize}
  \item[(i)] for each $i\geq 1$, $u_{i}$ is an \eve\ of $T$ associated
    to an \eva\ $\mu _{i}$ of $T$ where $|\mu _{i}|=1$ and the
    $\mu_{i}$'s are all distinct;

  \item[(ii)] $\textrm{sp}[u_{i}\textrm{ ; } i \geq 1]$ is dense in $X$;

  \item[(iii)] for any $i\geq 1$ and any $\varepsilon >0$, there exists an $n\not =i $ such that $||u_{n}-u_{i}||<\varepsilon $.
\end{itemize}
Then
$T$ has a \ps.

\end{theorem}

\subsection{Proof of Theorem \ref{th1bis}: the easy part}
We are first going to define the \op\ $T$, and show that it is bounded. We will then describe the unimodular \eve s of $T$, and show that $T$ satisfies the assumptions of Theorem \ref{th0bis}.
\par\smallskip
$\blacktriangleright$ \textbf{Construction of the \op\ $T$.} Let $(e_n)_{n\geq 1}$ denote the canonical basis of the space $\ell_2(\N)$ of complex square summable sequences. We denote by $H$  the space $\ell_{2}(\N)$. The construction depends on two sequences $(\lambda _{n})_{n\geq 1}$ and $(\omega _{n})_{n\geq 1}$ which will be suitably chosen  further on in the proof:
$(\lambda _{n})_{n\geq 1}$ is a sequence of unimodular complex numbers which are all distinct, and $(\omega _{n})_{n\geq 1}$ is a sequence of positive weights.
\par\smallskip
Let $j:\{2,+\infty \}\to \{1,+\infty \}$ be a function having the following two properties:

$\bullet$ for any $n\geq 2$, $j(n)<n$;

$\bullet$ for any $k\geq 1$, the set $\{n\geq 2 \textrm{ ; } j(n)=k\}$ is infinite (i.e. $j$ takes every value $k$ infinitely often).
\par\smallskip
Let $D$ be the diagonal \op\ on $H$ defined by $De_{n}=\lambda _{n}e_{n}$ for $n\geq 1$, and let $B$ be the weighted \bs\ defined by $Be_{1}=0$ and $Be_{n}=\alpha _{n-1}e_{n-1}$ for $n\geq 2$, where the weights $\alpha _{n}$, $n\geq 1$, are defined by
$$\alpha _{1}=\omega _{1}\,|\lambda _{2}-\lambda _{j(2)}|$$ and
$$\alpha _{n}=\omega _{n}\,\left|\frac{\lambda _{n+1}-\lambda _{j(n+1)}}{\lambda _{n}-\lambda _{j(n)}}\right|\quad \textrm{for any }n\geq 2.$$
This definition of $\alpha _{n}$ makes sense because $j(n)<n$, so that $\lambda _{n}\not =\lambda _{j(n)}$.
The \ops\ $D$ and $B$ being thus defined, we set $T=D+B$.
\par\medskip
$\blacktriangleright$ \textbf{Boundedness of the \op\ $T$.}
The first thing to do is to prove that
$T$ is indeed a bounded \op\ on $H$, provided some conditions on the $\lambda _{n}$'s and $\omega  _{n}$'s are imposed. The diagonal \op\ $D$ being obviously bounded, we have to figure out conditions for $B$ to be bounded. If $\gamma >0$ is fixed, we have $||B||\leq \gamma $ provided
$$\omega _{1}\,|\lambda _{2}-\lambda _{j(2)}|\leq \gamma
\quad \textrm{and} \quad\omega _{n-1}\,\left|\frac{\lambda _{n}-\lambda _{j(n)}}{\lambda _{n-1}-\lambda _{j(n-1)}}\right|\leq \gamma \quad \textrm{for any }n\geq 3.$$
If the weights $\omega _{n}>0$ are arbitrary, the $\lambda _{n}$'s can be chosen in such a way that these conditions are satisfied:

$\bullet$ $\omega _{1}>0$ is arbitrary, we take $\lambda _{1}=1$ for instance (we could start here from any $\lambda _{1}\in\T$);

$\bullet$ we have $j(2)=1$: take $\lambda _{2}$ such that $|\lambda _{2}-\lambda _{1}|\leq \gamma /\omega _{1}$ with $\lambda _{2}\not=\lambda _{1}$;

$\bullet$ take $\omega _{2}>0$ arbitrary;

$\bullet$ $j(3)\in\{1,2\}$: take $\lambda _{3}$ so close to $\lambda _{j(3)}$, $\lambda _{3}\not\in\{\lambda _{1},\lambda _{2}\}$, that $$|\lambda _{3}-\lambda _{j(3)}|\leq\frac{\gamma }{\omega _{2}}\,|\lambda _{2}-\lambda _{j(2)}|$$

$\bullet$ take $\omega _{3}>0$ arbitrary, etc.

Thus $||B||\leq \gamma $ provided $\lambda _{n}$ is so close to $\lambda _{j(n)}$ for every $n\geq 2$ that
$$|\lambda _{n}-\lambda _{j(n)}|\leq\frac{\gamma }{\omega _{n-1}}\,|\lambda _{n-1}-\lambda _{j(n-1)}|.$$
No condition on the $\omega _{n}$'s needs to be imposed there.
\par\medskip
$\blacktriangleright$ \textbf{Unimodular \eve s of the \op\ $T$.}
The algebraic equation $Tx=\lambda x$ with $x=\sum_{k\geq 1}x_{k}e_{k}$ is equivalent to the equations $\lambda _{k}x_{k}+\alpha _{k}x_{k+1}=\lambda x_{k}$, i.e.
$x_{k+1}=\frac{\lambda -\lambda _{k}}{\alpha _{k}}x_{k}$ for any $k\geq 1$, i.e.
$$x_{k}=\frac{(\lambda -\lambda _{k-1})\ldots (\lambda -\lambda _{1})}{\alpha _{k-1}\ldots\alpha _{1}}\,x_{1}.$$ Hence for any $n\geq 1$, the eigenspace $\ker(T-\lambda _{n})$ is $1$-dimensional and
$\ker(T-\lambda _{n})=\textrm{sp}[u^{(n)}]$, where
$$u^{(n)}=e_{1}+\sum_{k=2}^{n}\frac{(\lambda_{n} -\lambda _{k-1})\ldots (\lambda_{n} -\lambda _{1})}{\alpha _{k-1}\ldots\alpha _{1}}e_{k}.$$
Our aim is now to show the following lemma:
\begin{lemma}\label{lem1}
By choosing in a suitable way the coefficients $\omega _{n}$ and $\lambda _{n}$, it is possible to ensure that for any $n\geq 2$,
$$||u^{(n)}-u^{(j(n))}||\leq 2^{-n}$$ (the sequence $(2^{-n})_{n\geq 2}$ could be replaced by any sequence $(\gamma_{n})_{n \geq 2}$ going to zero with $n$).
\end{lemma}

\begin{proof}[Proof of Lemma \ref{lem1}]
We have:
\begin{eqnarray*}
 u^{(n)}-u^{(j(n))}&=&\sum_{k=2}^{j(n)}\left(\frac{(\lambda_{n} -\lambda _{k-1})\ldots (\lambda_{n} -\lambda _{1})}{\alpha _{k-1}\ldots\alpha _{1}}-\frac{(\lambda_{j(n)} -\lambda _{k-1})\ldots (\lambda_{j(n)} -\lambda _{1})}{\alpha _{k-1}\ldots\alpha _{1}}\right)e_{k}\\
 &+&\sum_{k=j(n)+1}^{n}\frac{(\lambda_{n} -\lambda _{k-1})\ldots (\lambda_{n} -\lambda _{1})}{\alpha _{k-1}\ldots\alpha _{1}}\,e_{k}:=v^{(n)}+w^{(n)}.
\end{eqnarray*}
We denote the first sum by $v^{(n)}$ and the second one by $w^{(n)}$. If $\varepsilon_{n} >0$ is any positive number, we can ensure that $||v^{(n)}||<\varepsilon _{n}$ by choosing $\lambda _{n}$ such that $|\lambda _{n}-\lambda _{j(n)}|$ is sufficiently small, because the quantities $\alpha _{k-1}\ldots\alpha _{1}$ for $k\leq j(n)$ do not depend on $\lambda_n$. Let us now estimate
\begin{eqnarray*}
 ||w^{(n)}||^{2}&=&\sum_{k=j(n)+1}^{n}\left|\frac{(\lambda_{n} -\lambda _{k-1})\ldots (\lambda_{n} -\lambda _{1})}{\alpha _{k-1}\ldots\alpha _{1}}\right|^{2}\\
 &=&\sum_{k=j(n)+1}^{n} \frac{1}{\omega _{k-1}^{2}\ldots\omega _{1}^{2}}\,\cdot\,\left|\frac{(\lambda_{n} -\lambda _{k-1})\ldots (\lambda_{n} -\lambda _{1})}{\lambda _{k}-\lambda _{j(k)}}\right|^{2}
\end{eqnarray*}
since $${\alpha _{k-1}\ldots\alpha _{1}}=\omega _{k-1}\ldots\omega _{1}\,|\lambda _{k}-\lambda _{j(k)}|.$$
We estimate now each term in this sum. We can suppose that $|\lambda _{p}-\lambda _{q}|\leq 1$ for any $p$ and $q$ (this is no restriction), so
$|(\lambda_{n} -\lambda _{k-1})\ldots (\lambda_{n} -\lambda _{1})|\leq |\lambda _{n}-\lambda _{j(n)}|$ since $j(n)\in\{1,\ldots, k-1\}$. Thus for $k=j(n)+1,\ldots, n$,
$$\frac{1}{\omega _{k-1}^{2}\ldots\omega _{1}^{2}}\,\cdot\,\left|\frac{(\lambda_{n} -\lambda _{k-1})\ldots (\lambda_{n} -\lambda _{1})}{\lambda _{k}-\lambda _{j(k)}}\right|^{2}\leq\frac{1}{\omega _{k-1}^{2}\ldots\omega _{1}^{2}}\,\cdot\,\left|\frac{\lambda _{n}-\lambda _{j(n)}}{\lambda _{k}-\lambda _{j(k)}}\right|^{2}\cdot$$
If $k\in\{j(n)+1,\ldots,n-1\}$, the term on the right-hand side can be made arbitrarily small provided that we choose $\lambda _{n}$ in such a way that $|\lambda _{n}-\lambda _{j(n)}|$ is very small \wrt\ the quantities $|\lambda _{k}-\lambda _{j(k)}|\,.\,\omega _{k-1}\ldots\omega _{1}$, $k<n$. However for $k=n$, we only get the bound ${\omega _{n-1}^{-2}\ldots\omega _{1}^{-2}},$ which has to be made small if we want $||w^{(n)}||$ to be small. So we have to impose a condition on the weights $\omega _{n}$: we take $\omega _{n-1}$ so large with respect to $\omega _{1},\ldots, \omega _{n-2}$ that ${\omega _{n-1}^{-2}\ldots\omega _{1}^{-2}}$ is extremely small.
\par\smallskip
All the conditions on the $\lambda _{n}$'s and the $\omega _{n}$'s needed until now can indeed be satisfied simultaneously: at  stage $n$ of the construction, we take
$\omega _{n-1}$ very large. After this we take $\lambda _{n}$ extremely close to $\lambda _{j(n)}$.
Thus we can ensure that $||w^{(n)}||<\varepsilon _{n}$, hence that $||u^{(n)}-u^{(j(n))}||<2\varepsilon _{n}$. Taking $\varepsilon _{n}=2^{-(n+1)}$ gives our statement.
\end{proof}

Thanks to Lemma \ref{lem1}, it is easy to see that $T$ satisfies the assumptions of Theorem \ref{th0bis}:

\begin{proposition}\label{prop1}
The \op\ $T$ satisfies the assumptions of Theorem \ref{th0bis}. Hence it has a \ps, and in particular its unimodular point spectrum is uncountable.
\end{proposition}

\begin{proof}[Proof of Proposition \ref{prop1}]
It suffices to show that the sequence $(u^{(n)})_{n\geq 1}$ satisfies properties (i), (ii) and (iii).
That (i) is satisfied is clear, since the vectors $u^{(n)}$ are \eve s of $T$ associated to the \eva s $\lambda _{n}$ which are all distinct. Since for each $n\geq 1$ the vector $u^{(n)}$ belongs to the span of the first $n$ basis vectors $e_{1},\ldots, e_{n}$ and $\pss{u^{(n)}}{e_{n}}\not = 0$, the linear span of the vectors $u^{(n)}$, $n\geq 1$, contains all finitely supported vectors of $\ell_{2}(\N)$, and thus (ii) holds true. It remains to prove (iii).
As the function $j$ takes every value in $\N$ infinitely often, it follows from Lemma \ref{lem1} that for any $n\geq 1$ there exists a strictly increasing sequence $(p_{s}^{(n)})_{s\geq 1}$ of integers such that
$$||u^{(p_{s}^{(n)})}-u^{(n)}|| \textrm{ tends to } 0 \textrm{ as } s \textrm{ tends to }+\infty .$$
Hence (iii) is true.
\end{proof}

In order to conclude the proof of Theorem \ref{th1}, it remains to show that $T$ is partially power-bounded \wrt\ $(n_{k})_{k\geq 0}$, with $\sup_{k\geq 0}||T^{n_{k}}||\leq 1+\delta $. This is the most difficult part of the proof, which uses the assumption that $(n_{k})_{k\geq 0}$ is not a \js, and it is the object of the next section.

\subsection{Proof of Theorem \ref{th1bis}: the hard part}
In order to estimate the norms $||T^{n_{k}}||$, we will show that provided the $\omega _{n}$'s and $\lambda _{n}$'s are suitably chosen, $||T^{n_{k}}-D^{n_{k}}||\leq \delta $ for every $k\geq 1$. Since $||D^{n_{k}}||=1$, this will prove that $||T^{n_{k}}||\leq 1+\delta $ for every $k\geq 1$.
\par\smallskip
$\blacktriangleright$ \textbf{An expression of $(T^{n}-D^{n})$.}
We first have to compute $(T^{n}-D^{n})x$ for $n\geq 1$ and $x\in H$. For $k,l\geq 1$, let $t_{k,l}^{(n)}=\pss{T^{n}e_{l}}{e_{k}}$ be the coefficient in row $k$ and column $l$ of the matrix representation of $T^{n}$. If $k>l$, $t_{k,l}^{(n)}=0$
(all coefficients below the diagonal are zero), and if $l-k>n$, $t_{k,l}^{(n)}=0$ (all coefficients which are not in one of the first $n$ upper diagonals of the matrix are zero). We have $t_{k,k}^{(n)}=\lambda _{k}^{n}$ for any $k\geq 1$.
\par\smallskip

\begin{lemma}\label{lem2}
For any $k,l\ge 1$ such that $1\leq l-k\leq n$,
$$t_{k,l}^{(n)}=\alpha _{l-1}\alpha _{l-2}\ldots\alpha _{k}\sum_{j_{k}+\ldots+j_{l}=n-(l-k)}\lambda _{k}^{j_{k}}\ldots\lambda _{l}^{j_{l}}.$$
\end{lemma}

\begin{proof}
The proof is done by induction on $n\geq 1$.

$\bullet$ $n=1$: in this case $l=k+1$, and the formula above gives $t_{k, k+1}^{(1)}=\alpha _{k}$, which is true.
\par\smallskip
$\bullet$ Suppose that the formulas above are true for any $m\leq n$. Let $k$ and $l$ be such that $1\leq l-k\leq n+1$ (in particular $l\geq 2$). We have
$$t_{k,l}^{(n+1)}=t_{k,l-1}^{(n)}t_{l-1,l}^{(1)}+t_{k,l}^{(n)}t_{l,l}^{(1)}=
\alpha _{l-1}t_{k,l-1}^{(n)}+\lambda _{l}t_{k,l}^{(n)}.$$
If $2\le l-k\le n$, we can apply the induction assumption to the two quantities $t_{k,l-1}^{(n)}$ and $t_{k,l}^{(n)}$, and we get
\begin{eqnarray*}
t_{k,l}^{(n+1)}&=& \alpha _{l-1}\alpha _{l-2}\ldots\alpha _{k}\sum_{j_{k}+\ldots+j_{l-1}=n-(l-1-k)}\lambda _{k}^{j_{k}}\ldots\lambda _{l-1}^{j_{l-1}}\\
&+&
 \alpha _{l-1}\alpha _{l-2}\ldots\alpha _{k}\sum_{j_{k}+\ldots+j_{l}=n-(l-k)}\lambda _{k}^{j_{k}}\ldots\lambda _{l-1}^{j_{l-1}}\lambda _{l}^{j_{l}+1}\\
& =&
 \sum_{j_{k}+\ldots+j_{l-1}+j_{l}=n+1-(l-k), \, j_{l}=0}\lambda _{k}^{j_{k}}\ldots\lambda _{l-1}^{j_{l-1}}\lambda _{l}^{j_{l}}\\
&+&
\sum_{j_{k}+\ldots+j_{l-1}+j_{l}=n+1-(l-k), \, j_{l}\geq 1}\lambda _{k}^{j_{k}}\ldots\lambda _{l-1}^{j_{l-1}}\lambda _{l}^{j_{l}}\\
&=&
\sum_{j_{k}+\ldots+j_{l-1}+j_{l}=n+1-(l-k)}\lambda _{k}^{j_{k}}\ldots\lambda _{l-1}^{j_{l-1}}\lambda _{l}^{j_{l}}
\end{eqnarray*}
and the formula is proved for $1\le l-k\le n$. It remains to treat the cases where $l-k=1$ and where $l-k=n+1$.
If $l-k=1$, we have $t_{k,k+1}^{(n+1)}=\alpha _{k}\lambda _{k}^{n}+\lambda _{k+1}t_{k,k+1}^{(n)}.$ By the induction assumption
$$t_{k,k+1}^{(n)}=\alpha _{k}\sum_{j_{k}+j_{k+1}=n-1}\lambda _{k}^{j_{k}}\lambda _{k+1}^{j_{k+1}}=\alpha _{k}\,\frac{\lambda _{k+1}^{n}-\lambda _{k}^{n}}{\lambda _{k+1}-\lambda _{k}}$$ so that
$$t_{k,k+1}^{(n+1)}=\alpha _{k}
\left(\lambda _{k}^{n}+\lambda _{k+1}\frac{\lambda_{k+1}^{n}-\lambda _{k}^{n}}{\lambda _{k+1}-\lambda _{k}}\right)=\alpha _{k}\,\frac{\lambda _{k+1}^{n+1}-\lambda _{k}^{n+1}}{\lambda _{k+1}-\lambda _{k}}=
\alpha _{k}\sum_{j_{k}+j_{k+1}=n}\lambda _{k}^{j_{k}}\lambda _{k+1}^{j_{k+1}}$$ which is the formula we were looking for.
Lastly, when $l-k=n+1$,
$t_{k,n+1+k}^{(n+1)}=\alpha _{n+k}t_{k,n+k}^{(n)}$. By the induction assumption $t_{k,n+k}^{(n)}=\alpha _{n+k-1}\ldots\alpha _{k}$, thus $t_{k,n+1+k}^{(n+1)}=\alpha _{n+k}\ldots\alpha _{k}$ and the formula is proved in this case too.
\end{proof}
\par\smallskip
$\blacktriangleright$ \textbf{A first estimate on the norms $||(T^{n}-D^{n})||$.}
For $x=\sum_{l\geq 1}x_{l}e_{l}$, let us estimate $||(T^{n}-D^{n})x||^{2}$: we have
$$(T^{n}-D^{n})x=\sum_{l\geq 1}x_{l}\left(\sum_{k\geq 1}t_{k,l}^{(n)}e_{k}\right)-
\sum_{l\geq 1}x_{l}t_{l,l}^{(n)}e_{l}=\sum_{l\geq 2}x_{l}\left(\sum_{k=\max(1,l-n)}^{l-1}t_{k,l}^{(n)}e_{k}\right)$$
so that by the Cauchy-Schwarz inequality
\begin{eqnarray*}
 ||(T^{n}-D^{n})x||^{2}\leq ||x||^{2}\sum_{l\geq 2}\left|\left| \sum_{k=\max(1, l-n)}^{l-1}t_{k,l}^{(n)}e_{k}\right|\right|^{2}
 \leq ||x||^{2} \sum_{l\geq 2}\sum_{k=\max(1,l-n)}^{l-1}|t_{k,l}^{(n)}|^{2}.
\end{eqnarray*}
We thus have to estimate for each $l\geq 2$ and $p\geq 1$ the quantities $$\sum_{k=\max(1,l-n_{p})}^{l-1}|t_{k,l}^{(n_{p})}|^{2}.$$
For $k,l\geq 1$, $1\leq l-k\leq n$, let
$$s_{k,l}^{(n)}=\sum_{j_{k}+\ldots+j_{l}=n-(l-k)}\lambda _{k}^{j_{k}}\ldots\lambda _{l}^{j_{l}}.$$ We have
$$t_{k,l}^{(n)}=\alpha _{l-1}\ldots\alpha _{k}\,s_{k,l}^{(n)}=\omega _{l-1}\ldots\omega _{k}\,\frac{|\lambda _{l}-\lambda _{j(l)}|}{|\lambda _{k}-\lambda _{j(k)}|}\,s_{k,l}^{(n)}$$
so that we have to estimate
$$\sum_{k=\max(1,l-n)}^{l-1}\omega _{l-1}^{2}\ldots\omega _{k}^{2}\,
\frac{|\lambda _{l}-\lambda _{j(l)}|^{2}}{|\lambda _{k}-\lambda _{j(k)}|^{2}}\,|s_{k,l}^{(n)}|^{2}.$$
We are going to show
 that the following property holds true:

\begin{lemma}\label{lem3}
For any $1\leq k\leq l-1$, there exists for each $k\le j\le l-1$ a complex number $c_{j}^{(k,l)}$ depending only on $\lambda _{1},\ldots,\lambda _{l-1}$ (and $k$ and $l$ of course), but neither on $\lambda _{l}$ nor on $n$, such that for any $n\geq l-k$,
$$s_{k,l}^{(n)}=\sum_{j=k}^{l-1}c_{j}^{(k,l)}\,\frac{\lambda _{l}^{n+1-(l-k)}-\lambda _{j}^{n+1-(l-k)}}{\lambda _{l}-\lambda _{j}}\cdot$$
\end{lemma}

\begin{proof}
The proof is again done by induction on $l\geq 2$.

$\bullet$ Let us first treat the case $l=2$: we have to show that there exists $c_{1}^{(1,2)}$ such that for any $n\geq 2$,
$$s_{1,2}^{(n)}=c_{1}^{(1,2)}\, \frac{\lambda _{2}^{n}-\lambda _{1}^{n}}{\lambda _{2}-\lambda _{1}}\cdot$$
But
$$s_{1,2}^{(n)}=\sum_{j_{1}+j_{2}=n-1}\lambda _{1}^{j_{1}}\lambda _{2}^{j_{2}}=\sum_{j_{1}=0}^{n-1}\lambda _{1}^{j_{1}}\lambda _{2}^{n-1-j_{1}}=\frac{\lambda _{2}^{n}-\lambda _{1}^{n}}{\lambda _{2}-\lambda _{1}}$$
so this holds true with $c_{1}^{(1,2)}=1$.
\par\smallskip
$\bullet$ Suppose that the property is true for some $l\geq 2$, and consider for $1\leq k \leq l$ and $n\geq l+1-k$ the quantities
\begin{eqnarray*}
s_{k,l+1}^{(n)}&=&
\sum_{j_{k}+\ldots+j_{l+1}=n-(l+1-k)}\lambda _{k}^{j_{k}}\ldots\lambda _{l+1}^{j_{l+1}}\\
&=&\sum_{j_{l+1}=0}^{n-(l+1-k)}\left(
\sum_{j_{k}+\ldots+j_{l}=n-(l+1+j_{l+1}-k)}\lambda _{k}^{j_{k}}\ldots\lambda _{l}^{j_{l}}
\right)\lambda _{l+1}^{j_{l+1}}.
\end{eqnarray*}
If $1\le k\le l-1$, we can apply the induction assumption and we get that
\begin{eqnarray*}
s_{k,l+1}^{(n)}&=& \sum_{j_{l+1}=0}^{n-(l+1-k)}\lambda _{l+1}^{j_{l+1}}s_{k,l}^{(n-1-j_{l+1})}\\
&=& \sum_{j_{l+1}=0}^{n-(l+1-k)} \lambda _{l+1}^{j_{l+1}}\sum_{j=k}^{l-1}c_j^{(k,l)}\left(\frac{\lambda _{l}^{n-j_{l+1}-(l-k)}-\lambda _{j}^{n-j_{l+1}-(l-k)}}{\lambda _{l}-\lambda _{j}}\right)
\end{eqnarray*}
where $c_j^{(k,l)}$ depends only on $\lambda _{1},\ldots, \lambda _{l-1}$ (the first equality in the display above comes from the fact that $j_{l+1} \leq n-1-l+k$, i.e. $l-k\leq n-1-j_{l+1}$).
Thus
\begin{eqnarray*}
s_{k,l+1}^{(n)}&=&
\sum_{j=k}^{l-1} \frac{c_j^{(k,l)}}{\lambda _{l}-\lambda _{j}}\left(\sum_{j_{l+1}=0}^{n-(l+1-k)}\lambda _{l+1}^{j_{l+1}}\lambda _{l}^{n-j_{l+1}-(l-k)}-\lambda _{l+1}^{j_{l+1}}\lambda _{j}^{n-j_{l+1}-(l-k)}\right)\\
&=&\sum_{j=k}^{l-1}\frac{c_j^{(k,l)}}{\lambda _{l}-\lambda _{j}}\left(\lambda _{l}^{n-(l-k)}\,
\frac{1-(\lambda _{l+1}\overline{\lambda }_{l})^{n-(l-k)}}{1-(\lambda _{l+1}\overline{\lambda }_{l})}
-\lambda _{j}^{n-(l-k)}\,
\frac{1-(\lambda _{l+1}\overline{\lambda }_{j})^{n-(l-k)}}{1-(\lambda _{l+1}\overline{\lambda }_{j})}
\right)\\
&=& \sum_{j=k}^{l-1}
\frac{c_j^{(k,l)}}{\lambda _{l}-\lambda _{j}}\left(\lambda _{l}\,\frac{\lambda _{l+1}^{n-(l-k)}-\lambda _{l}^{n-(l-k)}}{\lambda _{l+1}-\lambda _{l}}-
\lambda _{j}\,\frac{\lambda _{l+1}^{n-(l-k)}-\lambda _{j}^{n-(l-k)}}{\lambda _{l+1}-\lambda _{j}}\right)\\
&=&  \sum_{j=k}^{l-1}\left(-
\frac{\lambda_j c_j^{(k,l)}}{\lambda _{l}-\lambda _{j}}\right)\left(\frac{\lambda _{l+1}^{n-(l-k)}-\lambda _{j}^{n-(l-k)}}{\lambda _{l+1}-\lambda _{j}}\right)\\
&+&\left(\sum_{j=k}^{l-1}
\frac{c_j^{(k,l)}}{\lambda _{l}-\lambda _{j}}\lambda_l\right)\left(\frac{\lambda _{l+1}^{n-(l-k)}-\lambda _{l}^{n-(l-k)}}{\lambda _{l+1}-\lambda _{l}}\right)\\
\end{eqnarray*}
i.e
\begin{eqnarray*}
s_{k,l+1}^{(n)}
&=&\sum_{j=k}^{l}c_{j}^{(k,l+1)}\,\frac{\lambda _{l+1}^{n+1-(l+1-k)}-\lambda _{j}^{n+1-(l+1-k)}}{\lambda _{l+1}-\lambda _{j}}
\end{eqnarray*}
where $$c_{j}^{(k,l+1)}=-
\frac{\lambda_j c_j^{(k,l)}}{\lambda _{l}-\lambda _{j}}\quad \textrm{ for }k\leq j\leq l-1\quad \textrm{and}\quad
c_{l}^{(k,l+1)}=\sum_{j=k}^{l-1}
\frac{c_j^{(k,l)}}{\lambda _{l}-\lambda _{j}}\lambda_l
$$
depend only on $\lambda _{1},\ldots, \lambda _{l}$. This settles the case where $1\le k\le l-1$. When $k=l$, we get
$$s_{l,l+1}^{(n)}=\sum_{j_{l}+j_{l+1}=n-1}\lambda _{l}^{j_{l}}\lambda _{l+1}^{j_{l+1}}=\frac{\lambda _{l+1}^{n}-\lambda _{l}^{n}}{\lambda _{l+1}-\lambda _{l}},$$ and the statement is true with $c_{l}^{(l,l+1)}=1$.
\end{proof}

\par\smallskip
Let us now go back to the estimate on $||(T^{n_{p}}-D^{n_{p}})x||^{2}$, $p\geq 0$: we want to show that if the coefficients $\lambda _{l}$ are suitably chosen, the following holds true: for any $p\geq 0$ we have
\par\smallskip
$\bullet$ for any $2\leq l \leq n_{p}$, $$\sum_{k=1}^{l-1}|t_{k,l}^{(n_{p})}|^{2}\leq \delta ^{2}\,2^{-l},$$

$\bullet$ for any $ l \geq n_{p}+1$, $$\sum_{k=l-n_{p}}^{l-1}|t_{k,l}^{(n_{p})}|^{2}\leq \delta ^{2}\,2^{-l}.$$
Let us first consider the case $2\leq l \leq n_{p}$.
\par\smallskip
$\blacktriangleright$ \textbf{The ``easy'' estimate on $||(T^{n_{p}}-D^{n_{p}})||$.}
Let us write
\begin{eqnarray*}
 \sum_{k=1}^{l-1}|t_{k,l}^{(n_{p})}|^{2}&=&\sum_{k=1}^{l-1}{\omega _{l-1}^{2}\ldots\omega _{k}^{2}}
 \left|\frac{\lambda _{l}-\lambda _{j(l)}}{\lambda _{k}-\lambda _{j(k)}}\right|^{2}|s_{k,l}^{(n)}|^{2}\\
 &\leq&\sum_{k=1}^{l-1}{\omega _{l-1}^{2}\ldots\omega _{k}^{2}}
 \left|\frac{\lambda _{l}-\lambda _{j(l)}}{\lambda _{k}-\lambda _{j(k)}}\right|^{2}\left(\sum_{j=k}^{l-1}|c_{j}^{(k,l)}|\left|\frac{\lambda _{l}^{n_{p}+1-(l-k)}-\lambda _{j}^{n_{p}+1-(l-k)}}{\lambda _{l}-\lambda _{j}}\right|\right)^{2}.
\end{eqnarray*}
In the sum indexed by $j$, we have two different cases to consider: either $j=j(l)$ or $j\not =j(l)$. The case $j=j(l)$ can happen only when $j(l)\geq k$. Thus the sum can be decomposed as
\begin{eqnarray*}
 &&\sum_{k=1}^{j(l)}{\omega _{l-1}^{2}\ldots\omega _{k}^{2}}
 \left|\frac{\lambda _{l}-\lambda _{j(l)}}{\lambda _{k}-\lambda _{j(k)}}\right|^{2}
 \left(\sum_{j=k, j\not = j(l)}^{l-1}|c_{j}^{(k,l)}|\left|\frac{\lambda _{l}^{n_{p}+1-(l-k)}-\lambda _{j}^{n_{p}+1-(l-k)}}{\lambda _{l}-\lambda _{j}}\right|\right.\\
 &&\quad+\left.|c_{j(l)}^{(k,l)}|
 \left|\frac{\lambda _{l}^{n_{p}+1-(l-k)}-\lambda _{j(l)}^{n_{p}+1-(l-k)}}{\lambda _{l}-\lambda _{j(l)}}\right|
 \right)^{2}\\
 &&\quad + \sum_{k=j(l)+1}^{l-1}{\omega _{l-1}^{2}\ldots\omega _{k}^{2}}
 \left|\frac{\lambda _{l}-\lambda _{j(l)}}{\lambda _{k}-\lambda _{j(k)}}\right|^{2}
 \left(\sum_{j=k, j\not = j(l)}^{l-1}|c_{j}^{(k,l)}|\left|\frac{\lambda _{l}^{n_{p}+1-(l-k)}-\lambda _{j}^{n_{p}+1-(l-k)}}{\lambda _{l}-\lambda _{j}}\right|\right)^{2}\\
 \end{eqnarray*}
 which is less than
 \begin{eqnarray}
 \nonumber&&\quad 2\,\sum_{k=1}^{l-1}{\omega _{l-1}^{2}\ldots\omega _{k}^{2}}
 \left|\frac{\lambda _{l}-\lambda _{j(l)}}{\lambda _{k}-\lambda _{j(k)}}\right|^{2}
 \left(\sum_{j=k, j\not = j(l)}^{l-1}|c_{j}^{(k,l)}|\left|\frac{\lambda _{l}^{n_{p}+1-(l-k)}-\lambda _{j}^{n_{p}+1-(l-k)}}{\lambda _{l}-\lambda _{j}}\right|\right)^{2}\\
 \nonumber&&\quad+2\,\sum_{k=1}^{j(l)}{\omega _{l-1}^{2}\ldots\omega _{k}^{2}}\,.\,\frac{1}{|\lambda _{k}-\lambda _{j(k)}|^{2}}\,.\,|c_{j(l)}^{(k,l)}|^{2}\,.\,\left|\lambda _{l}^{n_{p}+1-(l-k)}-\lambda _{j(l)}^{n_{p}+1-(l-k)}\right|^{2}
 \end{eqnarray}
 and this in turn is less than
 \begin{eqnarray}
 \label{eq1} &&\quad |\lambda _{l}-\lambda _{j(l)}|^{2}\,\left(8\,\sum_{k=1}^{l-1}{\omega _{l-1}^{2}\ldots\omega _{k}^{2}}\,.\,\frac{1}{|\lambda _{k}-\lambda _{j(k)}|^{2}}\,.\,\left(\sum_{j=k, j\not = j(l)}^{l-1}|c_{j}^{(k,l)}|\frac{1}{|\lambda _{l}-\lambda _{j}|}
 \right)^{2}\right)\\
 \nonumber &&\quad +2\,\sum_{k=1}^{j(l)}{\omega _{l-1}^{2}\ldots\omega _{k}^{2}}\,.\,\frac{1}{|\lambda _{k}-\lambda _{j(k)}|^{2}}\,.\,|c_{j(l)}^{(k,l)}|^{2}\,.\,\left|\lambda _{l}^{n_{p}+1-(l-k)}-\lambda _{j(l)}^{n_{p}+1-(l-k)}\right|^{2}.
\end{eqnarray}
Suppose (as we may) that $\lambda _{l}$ is so close to $\lambda _{j(l)}$ that $$|\lambda _{l}-\lambda _{j(l)}|\leq\frac{1}{2}\min_{j\leq l-1, j\not =j(l)}|\lambda _{j}-\lambda _{j(l)}|.$$ Then for any $j\leq l-1$ with $j\not = j(l)$ we have
$|\lambda _{l}-\lambda _{j}|\geq |\lambda _{j}-\lambda _{j(l)}|-|\lambda _{l}-\lambda _{j(l)}|\geq\frac{1}{2}|\lambda _{j}-\lambda _{j(l)}|$.
Thus the  first term in the expression (\ref{eq1}) above is less than
\begin{eqnarray*}
32\,|\lambda _{l}-\lambda _{j(l)}|^{2}\,\sum_{k=1}^{l-1}{\omega _{l-1}^{2}\ldots\omega _{k}^{2}}\,.\,\frac{1}{|\lambda _{k}-\lambda _{j(k)}|^{2}}\,.\,\left(\sum_{j=k, j\not = j(l)}^{l-1}|c_{j}^{(k,l)}|\frac{1}{|\lambda _{j(l)}-\lambda _{j}|}
 \right)^{2}.
 \end{eqnarray*}
Since the quantity between the brackets depends only on $\lambda _{1},\ldots,\lambda _{l-1},\omega _{1},\ldots,\omega _{l-1}$ but not on $\lambda _{l}$, the expression in the display above can be made arbitrarily small if $|\lambda _{l}-\lambda _{j(l)}|$ is small enough. 
So we take, for any $l\geq 2$, $\lambda _{l}$ with
$|\lambda _{l}-\lambda _{j(l)}|$ so small that
$$32\,|\lambda _{l}-\lambda _{j(l)}|^{2}\sum_{k=1}^{l-1}{\omega _{l-1}^{2}\ldots\omega _{k}^{2}}\,.\,\frac{1}{|\lambda _{k}-\lambda _{j(k)}|^{2}}\,.\,\left(\sum_{j=k, j\not = j(l)}^{l-1}|c_{j}^{(k,l)}|\frac{1}{|\lambda _{j(l)}-\lambda _{j}|}
 \right)^{2}\leq \delta ^{2}2^{-(l+1)}.$$ Observe that the estimate we get here does not depend on $n_{p}$ (it is valid for any $n$ in fact). This takes care of the first term in the sum (\ref{eq1}).
 \par\smallskip
 $\blacktriangleright$ \textbf{The ``hard'' estimate on $||T^{n_{p}}-D^{n_{p}}||$.}
 We have now to estimate the term
 \begin{eqnarray}\label{eq4}
 2\,\sum_{k=1}^{j(l)}{\omega _{l-1}^{2}\ldots\omega _{k}^{2}}\,.\,\frac{1}{|\lambda _{k}-\lambda _{j(k)}|^{2}}\,.\,|c_{j(l)}^{(k,l)}|^{2}\,.\,\left|\lambda _{l}^{n_{p}+1-(l-k)}-\lambda _{j(l)}^{n_{p}+1-(l-k)}\right|^{2}.
 \end{eqnarray}
 We have
 \begin{eqnarray*}
 |\lambda _{l}^{n_{p}+1-(l-k)}-\lambda _{j(l)}^{n_{p}+1-(l-k)}|&\leq& |\lambda _{l}^{n_{p}}-\lambda _{j(l)}^{n_{p}}| +
 |\lambda _{l}^{l-k-1}-\lambda _{j(l)}^{l-k-1}|\\
 &\leq &|\lambda _{l}^{n_{p}}-\lambda _{j(l)}^{n_{p}}| + (l-k-1)
 |\lambda _{l}-\lambda _{j(l)}|\\
 &\leq& |\lambda _{l}^{n_{p}}-\lambda _{j(l)}^{n_{p}}| +(l-2)|\lambda _{l}-\lambda _{j(l)}|
 \end{eqnarray*}
so that the quantity in (\ref{eq4}) is less than
 \begin{eqnarray}
 \label{eq5} &&4|\lambda _{l}^{n_{p}}-\lambda _{j(l)}^{n_{p}}|^{2}\sum_{k=1}^{j(l)}{\omega _{l-1}^{2}\ldots\omega _{k}^{2}}\,.\,\frac{|c_{j(l)}^{(k,l)}|^{2}}{|\lambda _{k}-\lambda _{j(k)}|^{2}}\\
 \nonumber&&\qquad\qquad\qquad+4(l-2)^{2}|\lambda _{l}-\lambda _{j(l)}|^{2}\sum_{k=1}^{j(l)}{\omega _{l-1}^{2}\ldots\omega _{k}^{2}}\,.\,\frac{|c_{j(l)}^{(k,l)}|^{2}}{|\lambda _{k}-\lambda _{j(k)}|^{2}}\cdot
 \end{eqnarray}
 As previously the second term in this sum can be made arbitrarily small for any $l\geq 2$ provided $|\lambda _{l}-\lambda _{j(l)}|$ is sufficiently small, and we can ensure that it is less than $\delta ^{2}2^{-(l+2)}$. The difficult term to handle is the first one, and it is here that we use our assumption on the sequence $(n_{p})_{p\geq 0}$ (which was never used in the proof until this point). Our assumption is that for any $\varepsilon >0$ there exists a
 $\lambda \in\T\setminus\{1\}$ such that $\sup_{p\geq 0}|\lambda ^{n_{p}}-1|\leq\varepsilon $. This can be rewritten using the distance on $\T$ defined by $$d_{(n_{p})}(\lambda ,\mu )=\sup_{p\geq 0}|\lambda ^{n_{p}}-\mu ^{n_{p}}|$$ as: for any $\varepsilon >0$ there exists a
 $\lambda \in\T\setminus\{1\}$ such that $d_{(n_{p})}(\lambda ,1)\leq\varepsilon $. This means (see \cite{BaG2}) that there exists an uncountable subset $K$ of $\T$ such that $(K,d_{(n_{p})})$ is a \sep\ metric space. Thus $K$ contains a subset $K'$ which is uncountable and perfect for the distance $d_{(n_{p})}$. This means that for every $\varepsilon >0$ and every $\lambda \in K'$, there exists a $\lambda '\in K'$, $\lambda '\not=\lambda $ such that $d_{(n_{p})}(\lambda ,\lambda ')<\varepsilon $. Observe that since $n_{0}=1$, $|\lambda -\lambda '|\leq d_{(n_{p})}(\lambda ,\lambda ')<\varepsilon $.
 \par\smallskip
 In the construction of the $\lambda _{l}$'s, $l\geq 1$, we start by taking $\lambda _{1}\in K'$. Then we take $\lambda _{2}$ in $K'$ such that $d_{(n_{p})}(\lambda _{2},\lambda _{j(2)})$ is extremely small, which is possible since $\lambda _{j(2)}=\lambda _{1}$ is not isolated in $K'$. In the same way we can take the $\lambda _{l}$'s for $l\geq 2$ to be elements of $K'$ such that $d_{(n_{p})}(\lambda _{l},\lambda _{(l)})$ is arbitrarily small, $\lambda _{l}\not=\lambda _{j(l)}$. Then $|\lambda _{l}-\lambda _{j(l)}|$ is also arbitrarily small. 
 \par\smallskip
 With this suitable choice of the $\lambda _{l}$'s, we can estimate the remaining term in (\ref{eq5}):
\begin{eqnarray}
 \label{eq6}&&4\,|\lambda _{l}^{n_{p}}-\lambda _{j(l)}^{n_{p}}|^{2}\sum_{k=1}^{j(l)}{\omega _{l-1}^{2}\ldots\omega _{k}^{2}}\,.\,\frac{|c_{j(l)}^{(k,l)}|^{2}}{|\lambda _{k}-\lambda _{j(k)}|^{2}}\\
 \nonumber&&\qquad\qquad\qquad\leq
 4\,d_{(n_{p})}(\lambda _{l},\lambda _{j(l)})^{2}\sum_{k=1}^{j(l)}{\omega _{l-1}^{2}\ldots\omega _{k}^{2}}\,.\,\frac{|c_{j(l)}^{(k,l)}|^{2}}{|\lambda _{k}-\lambda _{j(k)}|^{2}}\cdot
\end{eqnarray}
Since the sum in $k$ depends only on $\lambda _{1},\ldots,\lambda _{l-1},\omega _{1},\ldots,\omega _{l-1}$, but not on $\lambda _{l}$, by taking $\lambda _{l}$ such that $d_{(n_{p})}(\lambda _{l},\lambda _{j(l)})$ is extremely small, we ensure that the right-hand term in (\ref{eq6}) is less than $\delta ^{2}2^{-(l+2)}$ for instance, for every $l\geq 2$.
\par\smallskip
Let us stress that the restrictions on the size of $|\lambda _{l}-\lambda _{j(l)}|$ and $d_{(n_{p})}(\lambda _{l},\lambda _{j(l)})$ are imposed for any $l\geq 2$, and do not depend on a particular $n_{p}$. Thus we have proved that for any $p\geq 1$ and any $2\leq l\leq n_{p}$, we have with these choices of $\lambda _{l}$
$$\sum_{k=1}^{l-1}|t_{k,l}^{(n_{p})}|^{2}\leq \delta ^{2}2^{-l}.$$
\par\smallskip
It remains to treat the case where $l\geq n_{p}+1$, where we have to estimate the quantity
$$\sum_{k=l-n_{p}}^{l-1}|t_{k,l}^{(n_{p})}|^{2}$$ which is less than
\begin{eqnarray}\label{eq7}
 \quad\sum_{k=l-n_{p}}^{l-1}{\omega _{l-1}^{2}\ldots\omega _{k}^{2}}\,.\,\left|\frac{\lambda _{l}-\lambda _{j(l)}}{\lambda _{k}-\lambda _{j(k)}}\right|^{2}
\left(\sum_{j=k}^{l-1}|c_{j}^{(k,l)}|\,
\left|\frac{\lambda _{l}^{n_{p}+1-(l-k)}-\lambda _{j}^{n_{p}+1-(l-k)}}{\lambda _{l}-\lambda _{j}}\right|
\right)^{2}.
\end{eqnarray}
We make the same decomposition as above in the sum in $j$, by separating the cases $j=j(l)$ and $j\not =j(l)$. The case $j=j(l)$ can only happen when $j(l)\geq k$, so when $j(l)\geq l-n_{p}$, i.e. $n_{p}\geq l-j(l)$. The estimates on the term not involving the index $j=j(l)$ are worked out exactly as previously, and this term can be made arbitrarily small provided
$|\lambda _{l}-\lambda _{j(l)}|$ is very small. The other term appears when $n_{p}\geq l-j(l)$, and is equal to
\begin{eqnarray}
\nonumber &&2
 \sum_{k=l-n_{p}}^{j(l)}
 |\lambda _{l}^{n_{p}+1-(l-k)}-\lambda _{j(l)}^{n_{p}+1-(l-k)}|^{2}
 {\omega _{l-1}^{2}\ldots\omega _{k}^{2}}\,.\,\frac{|c_{j(l)}^{(k,l)}|^{2}}{|\lambda _{k}-\lambda _{j(k)}|^{2}}\\
\nonumber &&\qquad\qquad\qquad\qquad\leq
 4\,d_{(n_{p})}(\lambda _{l},\lambda _{j(l)})^{2}
 \sum_{k=1}^{j(l)}{\omega _{l-1}^{2}\ldots\omega _{k}^{2}}\,.\,\frac{|c_{j(l)}^{(k,l)}|^{2}}{|\lambda _{k}-\lambda _{j(k)}|^{2}}\\
 \label{eq8}&&\qquad\qquad\qquad\qquad+4\,(l-2)^{2}|\lambda _{l}-\lambda _{j(l)}|^{2}\sum_{k=1}^{j(l)}{\omega _{l-1}^{2}\ldots\omega _{k}^{2}}\,.\,\frac{|c_{j(l)}^{(k,l)}|^{2}}{|\lambda _{k}-\lambda _{j(k)}|^{2}}\cdot
\end{eqnarray}
The reasoning is then exactly the same as previously, and if for each $l\geq 2$ the quantity $d_{(n_{p})}(\lambda _{l},\lambda _{j(l)})$ is sufficiently small we have for any $p\geq 1$ and any $l\geq n_{p}+1$ that
$$\sum_{k=l-n_{p}}^{l-1}|t_{k,l}^{(n_{p})}|^{2}\leq\delta ^{2}2^{-l}.$$ Hence $||T^{n_{p}}-D^{n_{p}}||\leq\delta $ for any $p\geq 1$. For $p=0$, $||T-D|| =||B||<\delta $, so that
$$\sup_{p\geq 0}||T^{n_{p}}-D^{n_{p}}||\leq \delta .$$ This proves that $T$ is partially power-bounded \wrt\ $(n_{p})_{p\geq 0}$, with the estimate $\sup_{p\geq 0}||T^{n_{p}}||\leq 1+\delta $, and this finishes the proof of Theorem \ref{th1bis}.

\begin{remark}
 It is not difficult to see that the \ops\ constructed in Theorem \ref{th1bis} are invertible: they are of the form $T=D+B$ , where $D$ is invertible with $||D||=1$, and $||B||$ can be made arbitrarily small in the construction.
\end{remark}

\section{Rigidity sequences}

\subsection{An abstract characterization of rigidity sequences}
As mentioned  already in the introduction, it is in a sense not
difficult to characterize \rs s via \mea s on the unit circle although
such a characterization is rather abstract and not so easy to handle
in concrete situations. This characterization is well-know, and hinted
at for instance in \cite{FW} or \cite{4}, but since we have been
unable to locate it precisely in the literature, we give below a quick
proof of it. A proof is also given in the preprint \cite{preprint}.
Here and later we denote by $\hat{\sigma}(n)$
the $n$-th Fourier coefficient of a measure $\sigma$.

\begin{proposition}\label{prop4}
 Let $(n_{k})_{k\geq 0}$ be a strictly increasing sequence of positive integers. The following assertions are equivalent:
 \begin{itemize}
  \item[(1)] there exists a dynamical system $\varphi$ on a \mea\ space $(X,\mathcal{F},\mu)$ which is \wmx\ and rigid \wrt\ $(n_{k})_{k\geq 0}$;

 \item[(2)] there exists a continuous \proba\ \mea\ $\sigma  $  on $\T$  such that $\hat{\sigma  }(n_{k})\to 1$ as $n_{k}\to +\infty $.
 \end{itemize}
\end{proposition}

Recall that a \proba\ \mea\ $\sigma  $ on $\T$ is said to be \emph{symmetric} if $\sigma  (\overline{A})=\sigma(A)$ for any Borel subset $A$ of $\T$ (where $\overline{A}$ denotes the conjugate set of $A$).
\par\smallskip
In order to prove Proposition \ref{prop4}, we are going to use the following well-known fact:

\begin{fact}\label{fact2}
Let $\varphi$ be a \mpt\  of the space $(X,\mathcal{F},\mu)$. The following
assertions are equivalent:
 \begin{itemize}
  \item[(a)] $\varphi$ is rigid \wrt\ $(n_{k})_{k\geq 0}$;

 \item[(b)] $U_{\varphi}^{n_{k}}\to I$ in the
weak operator topology
(WOT) of $L^{2}(X,\mathcal{F},\mu)$;

  \item[(c)] $U_{\varphi}^{n_{k}}\to I$ in the
strong operator topology
(SOT) of $L^{2}(X,\mathcal{F},\mu)$.
 \end{itemize}
\end{fact}

\begin{proof}[Proof of Fact \ref{fact2}]
The implication
$(c)\Rightarrow (b)$  is obvious. To prove $(b)\Rightarrow (a)$ it suffices to apply (b) to the characteristic functions $\chi_{A}$ of sets $A\in \mathcal{F}$:
$$\pss{U_{\varphi}^{n_{k}}\chi_{A}}{\chi_{A}}=\int_{X} \chi_{A}(\varphi^{n_{k}}(x))\chi_{A}(x)d\mu(x)\to \mu(A)\quad \textrm{as } n_{k}\to +\infty .$$
Now $\chi_{A\bigtriangleup \varphi^{-n_{k}}(A)}=\chi_{A}+\chi_{\varphi^{-n_{k}}(A)}-2\chi_{A}\chi_{\varphi^{-n_{k}}(A)}$ so that $$\mu({A\bigtriangleup \varphi^{-n_{k}}(A)})=2\mu(A)-2\int_{X} \chi_{A}(\varphi^{n_{k}}(x))\chi_{A}(x)d\mu(x)\to 0 \quad \textrm{as } n_{k}\to +\infty .$$ Hence $\varphi$ is rigid \wrt\ $(n_{k})_{k\geq 0}$. The proof of $(a)\Rightarrow (c)$ is done using the same kind of argument: thanks to the expression for $\chi_{A\bigtriangleup \varphi^{-n_{k}}(A)}$, we get that $||U_{\varphi}^{n_{k}}\chi_{A}-\chi_{A}||_{L^{2}}\to 0$ as $n_{k}\to +\infty $ for any $A\in \mathcal{F}$. Hence $||U_{\varphi}^{n_{k}}f-f||_{L^{2}}\to 0$ for any $f\in L^{2}(X,\mathcal{F},\mu)$, which is assertion $(c)$.
\end{proof}

\begin{proof}[Proof of Proposition \ref{prop4}]
Suppose first that (1) holds true, and let $\sigma  _{0}$ be the
reduced maximal spectral type of $U_{\varphi}$, i.e. the maximal
spectral type of the unitary \op\ $U$ induced by $U_{\varphi}$ on the
subspace
$H_{0}=\{f\in L^{2}(X,\mathcal{F},\mu) \textrm{ ; } \int_{X}
f(x)d\mu({x})=0\}$ (which is \inv\ by $U_{\varphi}$).
Note that $U_\varphi$ is unitary by Remark \ref{rem:inv}.
For the definition and basic properties of the spectral type of a
unitary operator see e.g.\cite{Na}.
Let $f_{0}\in H_{0}$ with $||f_{0}||=1$ be such that the spectral measure $\sigma  _{f_{0}}$ of $f_{0}$ is a representant of the class $\sigma  _{0}$. Since $\varphi$ is \wmx, $\sigma  _{f_{0}}$ is continuous, and it is a \proba\ \mea\ since $||f_{0}||=1$. For any $n\in\Z$ we have
$$\pss{U_\varphi^{n}f_{0}}{f_{0}}=\int_{\T}\lambda ^{n}d\sigma_{f_{0}}(\lambda)=\hat{\sigma  }_{f_{0}}(n).$$ Since $||U_{\varphi}^{n_{k}}f-f||_{L^{2}}\to 0$ for any $f\in H_{0}$, we get in particular that $\hat{\sigma  }_{f_{0}}(n_{k})\to ||f_{0}||^{2}=1$ as $n_{k}\to +\infty $, so (2) holds true.
\par\smallskip
Conversely, let $\sigma  $ be a continuous \proba\ \mea\ on $\T$ such that $\hat{\sigma  }(n_{k})\to 1$. Then $\int_{\T}|\lambda ^{n_{k}}-1|^{2}d\sigma  (\lambda )\to 0$ as $n_{k} \to +\infty $, so that in particular we have $\int_{\T}|\lambda ^{n_{k}}-1|d\sigma  (\lambda )\to 0$. Indeed $|\lambda ^{n_{k}}-1|^{2}=2(1-\Re e (\lambda ^{n_{k}}))$, hence
$\int_{\T}|\lambda ^{n_{k}}-1|^{2}d\sigma  (\lambda )=2\Re e (1-\hat{\sigma  }(n_{k}))\to 0$.
Let now $\overline{\sigma  }$ be the probability \mea\ on $\T$ defined by $\overline{\sigma  }(A)=\sigma  (\overline{A})$ for any $A\in \mathcal{F}$. Then $\overline{\sigma  }$ is a continuous \mea\ on $\T$, and
$\int_{\T}|\lambda ^{n_{k}}-1|d\overline{\sigma  }  (\lambda )
=\int_{\T}|\lambda ^{n_{k}}-1|d{\sigma  }  (\lambda )$ so that in particular $\hat{\overline{\sigma  }}(n_{k})\to 1$ as $n_{k} \to +\infty $. Setting $\rho :=(\sigma  +\overline{\sigma  })/2$, we obtain a continuous \emph{symmetric} probability \mea\ on $\T$ such that $\hat{\rho }(n_{k})\to 1$. So we can assume without loss of generality that the measure $\sigma  $ given by (2) is symmetric, and we have to construct out of this a \wmx\ \mpt\ of a \proba\ space which is \rg\ \wrt\ $(n_{k})$. The class of transformations which we use for this is the class of stationary Gaussian processes. We refer the reader to one of
 the references \cite{CFS}, \cite{Pet2} or \cite[Ch. 8]{Pe} for details about this, and in the forthcoming proof we use the notations of \cite[Ch. 8]{Pe}. Since $\sigma  $ is a symmetric probability \mea\ on $\T$, there exists a real-valued stationary Gaussian process $(X_{n})_{n\in\Z}$ whose spectral \mea\ is $\sigma  $. This Gaussian process lives on a \proba\
 space $(\Omega , \Sigma , \P)$ which can be realized as a sequence space: $\Omega =\R^{\Z}$, $\Sigma $ is the $\sigma  $-algebra generated by the sets $\Theta _{m,A}=\{(\omega _{n})_{n\in\Z} \textrm{ ; } (\omega _{-m},\ldots,\omega _{m})\in A\}$, $m\geq 0$, $A$ is a Borel subset of $\R^{2m+1}$, and
$\P$ is the probability that $(X _{-m},\ldots,X _{m})$ belongs to $A$, where $\P$ is $\tau$-invariant for the shift $\tau$ of the space $\R^\Z$, and $(X_n)$ satisfy $X_{n+1}=X_{n}\circ \tau$ for any $n\in\Z$.
The 
shift $\tau $ defines a \wmx\ transformation of $(\Omega , \Sigma , \P)$
 (see \cite{Pet2} for instance), and we have to see that it is \rg\ \wrt\ the sequence $(n_{k})$. Using the same argument as in the proof of \cite[Ch. 8, Th. 3.2]{Pe}, it suffices to show that for any functions $f,g$ belonging to $\mathcal{G}_{c}$, the complexification of the Gaussian subspace  of $L^{2}(\Omega , \Sigma , \P)$ spanned by $X_{n}, n\in\Z$, we have $\pss{U_{\tau}^{n_{k}}f-f}{g}\to 0$ as $n_{k}\to +\infty $.
If $\Phi:\mathcal{G}_{c}\to L^{2}(\T,\sigma  )$  denotes the map defined on the linear span of the $X_{n}$'s by
$\Phi(\sum c_{n} X_{n}):=\sum c_{n}\lambda ^{n}$, then $\Phi$ extends to a surjective isometry of $\mathcal{G}_{c}$ onto $L^{2}(\T,\sigma  )$, and we have for any $f\in \mathcal{G}_{c}$ that $U_{\tau}f=(\Phi^{-1}\circ M_{\lambda }\circ \Phi )f$, where $M_{\lambda }$ denotes multiplication by the independent variable $\lambda $ on $L^{2}(\T,\sigma  )$. Thus
$$\pss{U_{\tau}^{n_{k}}f-f}{g}=\pss{M_{\lambda }^{n_{k}}\Phi f-\Phi f}{\Phi g}=\int_{\T}(\lambda ^{n_{k}}-1)(\Phi f)(\lambda )\overline{(\Phi g (\lambda ))}d\sigma  (\lambda ).$$
Now if $h$ is any function in $L^{1}(\T,\sigma  )$, we have that $\int_{\T}|\lambda ^{n_{k}}-1|\,|h(\lambda )|d\sigma  (\lambda )\to 0$ as $n_{k}\to +\infty $ (it suffices to approximate $h$ by functions $h'\in L^{\infty }(\T,\sigma  )$ in $L^{1}(\T,\sigma  )$). Since $(\Phi f)\overline{(\Phi g )}$ belongs to $L^{1}(\T,\sigma  )$, we get that $\pss{U_{\tau}^{n_{k}}f-f}{g}\to 0$. It follows that $U_{\tau}^{n_{k}}\to I$ in the WOT of $L^{2}(\Omega ,\Sigma  ,\P)$ and $\tau$ is \rg\ \wrt\ $(n_{k})$ by Fact \ref{fact2}.
\end{proof}

\begin{remark}
The \ga\ dynamical systems considered in the proof of Proposition \ref{prop4} live on the space of sequences $\R^{\Z}$, which is not compact. But by the Jewett-Krieger Theo\-rem (see for instance \cite{Pet}), such a system is metrically isomorphic to a homeomorphism of the Cantor set.
\end{remark}

\subsection{Examples of rigidity and non-rigidity sequences}
Our first example of \rs s (obtained also in \cite{preprint}) is the following:
\begin{example}\label{ex1}
 Let $(n_{k})_{k\geq 0}$ be a strictly increasing sequence of positive numbers such that ${n_{k+1}}/{n_{k}}$ tends to infinity. Then $(n_{k})_{k\geq 0}$ is a \rs.
\end{example}

This fact follows from the following stronger Proposition
\ref{prop3} below, which will allow us to show later on in the paper that any such sequence is a \urs\ in the linear framework. The proof of Proposition \ref{prop3} uses ideas from \cite{BaG2}.

\begin{proposition}\label{prop3}
 Let $(n_{k})_{k\geq 0}$ be a strictly increasing sequence of positive numbers such that ${n_{k+1}}/{n_{k}}$ tends to infinity. There exists a compact perfect subset $K$ of $\T$ having the following two properties:
 \begin{itemize}
  \item[(i)] for any $\varepsilon >0$ there exists a compact perfect subset $K_{\varepsilon }$ of $K$ such that for any $\lambda \in K_{\varepsilon }$,
$$\sup_{k\geq 0}|\lambda ^{n_{k}}-1|\leq \varepsilon;$$

  \item[(ii)] $\lambda ^{n_{k}}$ tends to $1$ uniformly on $K$.
 \end{itemize}
\end{proposition}
Note that the existence of a compact perfect subset $K$ of $\T$ satisfying (ii) above implies that $(n_k)_{k\geq 0}$ is a rigidity sequence. Indeed, any continuous probability measure $\sigma$ supported on $K$ satifies assertion (2) in Proposition \ref{prop4}.
\begin{proof}
For any $k\geq 1$, let $\gamma _{k}=5\pi\,\sup_{j\geq k}({n_{j-1}}/{n_{j}})$: $\gamma _{k}$ decreases to $0$ as
$k$ tends to infinity, and let $k_{0} $ be such that $\gamma _{k}\leq \frac{1}{2}$ for any $k\geq k_{0}$. Let $\theta _{k}\in ]0,\frac{\pi}{2}[$ be such that $\gamma _{k}=\sin \theta _{k}$ for $k\geq k_{0}$. The sequence $(\theta _{k})$ decreases to $0$, and $\theta _{k}\sim \gamma _{k}$ as $k$ tends to infinity. Thus there exists a $k_{1}\geq k_{0}$ such that for any $k\geq k_{1}$,
$\theta _{k}\geq 4\pi\, \sup_{j\geq k}( n_{j-1}/n_{j})\geq 4\pi\,(n_{k}/n_{k+1})$, so that $({n_{k+1}}/{n_{k}})\, .\,{\theta _{k}}\geq {4\pi}$. Let $$K_{0}=\{\lambda \in\T \textrm{ ; } \forall k\geq k_{1}\quad |\lambda ^{n_{k}}-1|\leq 2\gamma _{k}\}.$$ If we write $\lambda\in\T$ as $\lambda =e^{2i\theta }$, $\theta \in [0,\pi[$, $\lambda $ belongs to $K_{0}$ \ifff\ $|\sin(n_{k}\theta )|\leq \gamma _{k}$ for any $k\geq k_{1}$. Let $F_{k}=\{\theta \in [0,\pi[ \textrm{ ; } |\sin(n_{k}\theta )|\leq \sin \theta_{k} \}$: $F_{k}$ consists of intervals of the form $[-\frac{\theta _{k}}{n_{k}}+\frac{l\pi}{n_{k}},\frac{\theta _{k}}{n_{k}}+\frac{l\pi}{n_{k}}]
 $, $l\in\Z$.
We will construct a Cantor subset $K$ of $K_{0}$ as $K=\bigcap_{k\geq k_{1}}\bigcup_{j\in I_{k}}J_{j}^{(k)}$ where the arcs $J_{j}^{(k)}$ have the form
\begin{eqnarray}
 \label{eq9} J_{j}^{(k)}=\left\{e^{i\theta } \textrm{ ; } \theta \in \left[-\frac{\theta _{k}}{n_{k}}+\frac{l_{j}^{(k)}\pi}{n_{k}},\frac{\theta _{k}}{n_{k}}+\frac{l_{j}^{(k)}\pi}{n_{k}}\right]\right\}
\end{eqnarray}
for some $l_{j}^{(k)}\in\Z$.
Observe that such arcs are disjoint as soon as $\frac{2\theta _{k}}{n_{k}}<\frac{\pi}{n_{k}}$, i.e. $\theta _{k}<\frac{\pi}{2}$, which is indeed the case, and that the arc corresponding to $l_{j}^{(k)}=0$ contains the point $1$ in its interior. There are $2n_{k}$ such intervals. We are going to construct by induction on $k$ a collection $(J_{j}^{(k)})_{j\in I_{k}}$ in such a way that each $J_{j}^{(k)}$ has the form  given in (\ref{eq9}) and is contained in an arc of the collection $(J_{j}^{(k-1)})_{j\in I_{k-1}}$ constructed at step $k-1$, and the collection $(J_{j}^{(k)})_{j\in I_{k}}$ contains the arc $[-\frac{\theta _{k}}{n_{k}},\frac{\theta _{k}}{n_{k}}]$ corresponding to the case $l=0$.
 We start for $k=k_{1}$ with the collection of all the $2n_{1}$ arcs above. Suppose that the arcs at step $k$ are constructed, and write one of them as
 $$J_{j}^{(k)}=\left[-\frac{\theta _{k}}{n_{k}}+\frac{l_{j}^{(k)}\pi}{n_{k}},\frac{\theta _{k}}{n_{k}}+\frac{l_{j}^{(k)}\pi}{n_{k}}\right].$$ Let us look for arcs of the form
 $$\left[-\frac{\theta _{k+1}}{n_{k+1}}+\frac{r\pi}{n_{k+1}},\frac{\theta _{k+1}}{n_{k+1}}+\frac{r\pi}{n_{k+1}}\right],\; r\in\Z,$$
 contained in $J_{j}^{(k)}$. There are
 $\lfloor \frac{1}{\pi} (\frac{n_{k+1}}{n_{k}}{\theta _{k}}-{\theta _{k+1}})\rfloor=p_{k+1}$ such intervals contained in
 $J_{j}^{(k)}$. By construction $p_{k+1}\geq \frac{1}{\pi} (4\pi-\frac{\pi}{2})-1 \geq 2$. Remark that in the case where $J_{j}^{(k)}$ is the arc $\{e^{i\theta }\textrm{ ; } \theta \in[-\frac{\theta _{k}}{n_{k}},\frac{\theta _{k}}{n_{k}}]\}$ we have in the collection $(J_{j}^{(k+1)})$ the arc
$\{e^{i\theta }\textrm{ ; } \theta \in[-\frac{\theta _{k+1}}{n_{k+1}},\frac{\theta _{k+1}}{n_{k+1}}]\}$ (which is indeed contained in the arc $\{e^{i\theta }\textrm{ ; } \theta \in[-\frac{\theta _{k}}{n_{k}},\frac{\theta _{k}}{n_{k}}]\}$). We obtain in this fashion a perfect  Cantor set $K$, which contains the point $1$ by construction, such that $\lambda ^{n_{k}}$ tends to $1$ uniformly on $K$ (as $|\lambda ^{n_{k}}-1|\leq 2\gamma _{k}$ for any $\lambda \in K$ and any $k\geq k_{1}$). Let $\varepsilon >0$. There exists an integer $\kappa $ such that for any $k\geq \kappa $ and any $\lambda \in K$, $|\lambda ^{n_{k}}-1|\leq \varepsilon $. Since $1$ belongs to $K$, the set $K_{\varepsilon }=\{\lambda \in K \textrm{ ; } |\lambda -1|\leq \varepsilon /n_{\kappa -1}\}$ is a compact perfect subset of $\T$, and for any $\lambda \in K_{\varepsilon }$ and any $0\leq k\leq \kappa -1$, $$|\lambda ^{n_{k}}-1|\leq \frac{\varepsilon }{n_{\kappa -1}}n_{k}\leq\varepsilon .$$ Hence $\sup_{k\geq 0}|\lambda ^{n_{k}}-1|\leq\varepsilon $ for any $\lambda \in K_{\varepsilon }$, and Proposition \ref{prop3} is proved.
\end{proof}

\begin{remark}\label{rem1}
We have shown at the end of the proof of Proposition \ref{prop3} that
if $K$ is a compact perfect subset of $\T$ such that $\lambda ^{n_{k}}$ tends to
$1$ uniformly on $K$, and if $K$ contains the point $1$, then for any
$\varepsilon>0$ there exists a $\lambda\in K\setminus\{1\}$ such that
$\sup_{k\geq 0}|\lambda ^{n_{k}}-1|\leq\varepsilon $. If we do not
suppose that $K$ contains the point $1$, the set
$\tilde{K}=\{\lambda \overline{\mu} \textrm{ ; } \lambda ,\mu \in K\}$ is compact, perfect, contains the point $1$, and $\lambda ^{n_{k}}$ still tends to $1$ uniformly on $\tilde{K}$. We thus have the following fact, which we record here for further use:

\begin{fact}\label{fact0}
The following assertions are equivalent:
\begin{itemize}
 \item[(i)] there exists a compact perfect subset $K$  of $\T$ such that $\lambda ^{n_{k}}$ tends to $1$ uniformly on $K$;

 \item[(ii)] there exists a compact perfect subset $K $ of $\T$ such that $\lambda ^{n_{k}}$ tends to $1$ uniformly on $K$, and for any $\varepsilon >0$ there exists a $\lambda \in K\setminus\{1\}$ such that $\sup_{k\geq 0}|\lambda ^{n_{k}}-1|\leq\varepsilon$.
\end{itemize}
\end{fact}
\end{remark}

Our next examples concern sequences $(n_{k})_{k\geq 0}$ such that $n_{k}$ divides $n_{k+1}$ for any $k\geq 0$ (we write this as $n_{k}|n_{k+1}$). We begin with the case where $\limsup_{k\to\infty } {n_{k+1}}/{n_{k}}=+\infty $, since in this case we can derive a stronger conclusion. Recall \cite{BaG1} that such sequences are Jamison sequences.

\begin{proposition}\label{prop22}
 Let $(n_{k})_{k\geq 0}$ be a sequence such that $n_{k}|n_{k+1}$ for any $k\geq 0$ and $\limsup_{k\to\infty } {n_{k+1}}/{n_{k}}=+\infty $. There exists a compact perfect subset $K$ of $\T$ containing the point $1$ such that $\lambda ^{n_{k}}\to 1$ uniformly on $K$.
\end{proposition}

\begin{proof}
 Since $n_{k}|n_{k+1}$ for any $k\geq 0$, we have $n_{k+1}\geq 2n_{k}$, so that $$\sum_{k\geq 1}\frac{1}{n_{k}}\leq 1 \quad \textrm{and}\quad \sum_{j\geq k+1}\frac{1}{n_{j}}\leq \frac{2}{n_{k+1}} \quad \textrm{for any}\quad  k\geq 1.$$ Let $(k_{p})_{p\geq 1} $ be a strictly increasing sequence of integers such that $\frac{n_{k_{p}}}{n_{k_{p}-1}}\to +\infty $ as $k_{p}\to\infty $. For any sequence $\varepsilon \in\{0,1\}^{\N}$ of zeroes and ones, $\varepsilon =(\varepsilon _{p})_{p\geq 1}$, consider the real number of $[0,1]$
 $$\theta _{\varepsilon }=\sum_{p\geq 1}\frac{\varepsilon _{p}}{n_{k_{p}}}, \quad \textrm{ and } \quad \lambda _{\varepsilon }=e^{2i\pi\theta _{\varepsilon }}\in\T.$$
 The set $K=\{\lambda _{\varepsilon } \textrm{ ; }\varepsilon \in \{0,1\}^{\N}\}$ is compact, perfect, and contains the point $1$. Let us now show that $\lambda _{\varepsilon }^{n_{k}}$ tends to $1$ uniformly \wrt\ $\varepsilon \in\{0,1\}^{\N}$. Fix $\delta >0$, and let $p_{0}\geq 1$ be such that for any $p\geq p_{0}$, $\frac{n_{k_{p}-1}}{n_{k_{p}}}<\frac{\delta }{4\pi}$. 
 Let $k\geq k_{p_{0}}$, and $\varepsilon \in\{0,1\}^{\N}$. There exists a $p\geq p_{0}$ such that $n_{k_{p}}\leq n_{k}\leq n_{k_{p+1}-1}$. We have $$n_{k}\theta _{\varepsilon }=n_{k}\sum_{j=1}^{p}\frac{\varepsilon _{j}}{n_{k_{j}}}+n_{k}\sum_{j\geq p+1}\frac{\varepsilon_{j}}{n_{k_{j}}}\cdot$$
 Since $n_{k_{j}}|n_{k}$ for any $j=1,\ldots,p$, $n_{k}\sum_{j=1}^{p}\frac{\varepsilon _{j}}{n_{k_{j}}}$ belongs to $\Z$. Hence
 $$|e^{2i\pi n_{k}\theta _{\varepsilon }}-1|\leq 2\pi\,n_{k}\sum_{j\geq p+1}\frac{1}{n_{k_{j}}}\leq
 2\pi\,n_{k}\sum_{j\geq k_{p+1}}\frac{1}{n_{{j}}}\leq
 4\pi\,\frac{n_{k}}{n_{k_{p+1}}}\leq 4\pi\,\frac{n_{k_{p+1}-1}}{n_{k_{p+1}}}<\delta ,$$ so
 $|\lambda _{\varepsilon }^{n_{k}}-1|<\delta $ for any $k\geq k_{p_{0}}$ and $\varepsilon \in\{0,1\}^{\N}$. This proves our statement.
\end{proof}

Let us now move over to the case where $n_{k}|n_{k+1}$ for every $k\geq 0$, but where ${n_{k+1}}/{n_{k}}$ is possibly bounded: for instance $n_{k}=2^{k}$ for any $k\geq 0$. Is $(n_{k})_{k\geq 0}$ a \rs? Somewhat surprisingly, the answer is yes. This was kindly shown to us by Jean-Pierre Kahane, who proved the following proposition:

\begin{proposition}\label{prop33}
 Let $(n_{k})_{k\geq 0}$ be a sequence such that  $n_{k}|n_{k+1}$ for every $k\geq 0$. Then $(n_{k})_{k\geq 0}$ is a \rs.
\end{proposition}

This proposition is also proved in the preprint \cite{preprint}.

\begin{proof}
 Let $(a_{k})_{k\geq 1}$ be a decreasing sequence of positive numbers going to $0$ as $k$ goes to infinity, with $a_k<1$ for every $k\geq 1$, such that the series $\sum_{k\geq 1}a_{k}$ is divergent. Consider the infinite convolution of Bernoulli measures defined on $[0,2\pi]$ by
 $$\mu=\ast_{j\geq 1} ((1-a_{j})\delta _{0}+a_{j}\delta _{\frac{1}{n_{j}}}),$$ where $\delta _{a}$ denotes the Dirac measure at the point $a$ for any $a\in [0,2\pi]$.
 Clearly $\mu $ is a \proba\ \mea\ on $[0,2\pi]$ which is continuous. Indeed $\mu $ is the distribution of the random variable
 $$\xi=\sum_{j\geq 1} \frac{\varepsilon _{j}}{n_{j}},$$ where $(\varepsilon _{j})_{j\geq 1}$ is a sequence of independent Bernoulli random variables taking values $0$ and $1$ with probabilities $p_{0j}=1-a_{j}$ and $p_{1j}=a_{j}$ respectively.
Since $\sum a_j=+\infty$, the \mea\ $\mu$ is continuous by a result of L\'evy (see \cite{GK} for a simple proof).
 It thus remains to prove that
 $\hat{\mu}(n_{k})\to 1 $ as $n_{k}\to +\infty $. Since $n_{j}|n_{j+1}$ for each $j\geq 0$,
 $$\hat{\mu }(n_{k})=\prod_{j\geq k+1}(1-a_{j}+a_{j}e^{2i\pi \frac{n_{k}}{n_{j}}})=
 \prod_{j\geq k+1}(1-a_{j}(1-e^{2i\pi\frac{n_{k}}{n_{j}}})).$$
Recall now the following easy fact: for any $N\geq 1$ and any complex numbers $x_j$ with $|x_j|\le 1$ for every $j=1,\ldots,N$, we have $|\prod_{j=1}^N x_j-1|\le\sum_{j=1}^N|x_j-1|.$ Since for any $j\geq k+1$,
$$|1-a_{j}(1-e^{2i\pi \frac{n_{k}}{n_{j}}})|=|1-a_j+a_{j}e^{2i\pi \frac{n_{k}}{n_{j}}}|\le 1-a_j+a_j=1,$$ we get that
 $$|\hat{\mu}(n_{k})-1|\leq \sum_{j\geq k+1}a_{j}|1-e^{2i\pi\frac{n_{k}}{n_{j}}}|\leq 2\pi\, a_{k+1} \sum_{j\geq k+1}
 \frac{n_{k}}{n_{j}}\leq 4\pi\, a_{k+1}$$ since the sequence $(a_{j})_{j\geq 1}$ is decreasing and $n_{k}\sum_{j\geq k+1}
 \frac{1}{n_{j}}\leq n_{k+1}\sum_{j\geq k+1}
 \frac{1}{n_{j}}\leq 2$, as seen in Proposition \ref{prop22} above. Hence $\hat{\mu}(n_{k})\to 1 $, and this proves Proposition \ref{prop33}.
 \end{proof}

\begin{remark}\label{rem3}
Remark that if $n_{k}=2^{k}$ for instance, the only $\lambda$'s in $\T$  such that $\lambda ^{n_{k}}$ tends to $1$
are the ${2^{k}}^{th}$ roots of $1$. More generally, it is not difficult to see that if $n_{k}|n_{k+1}$ and $\sup n_{k+1}/n_k$ is finite, $\lambda ^{n_{k}} \to 1$ \ifff\ there exists a $k_{0}$ such that $\lambda ^{n_{k_{0}}}=1$.
\end{remark}

\begin{remark}\label{rem33}
 The proof of Proposition \ref{prop33} yields a bit more, namely that given any sequence $(a_{k})_{k\geq 0}$  of positive numbers decreasing to zero and such that the series $\sum a_{k}$ diverges, there exists a continuous probability \mea\ $\sigma  $ on $\T$ such that $|\hat{\sigma  }({n_{k}})-1|\le a_{k}$ for every $k\geq 0$. This will turn out to be crucial in the proof of the statement of Example \ref{ex77}. In general one cannot obtain such a \mea\ $\sigma  $ with $\sum |\hat{\sigma  }({n_{k}})-1|<+\infty $: this would imply that the series $\sum|\lambda^{{n_{k}}}-1|$ converges $\sigma  $-a.e., so that $|\lambda^{{n_{k}}}-1|\to 0$ $\sigma  $-a.e., and we have seen in Remark \ref{rem3} above that this is impossible if $n_{k+1}/n_{k}$ is bounded for instance.
\end{remark}

The proof of Proposition \ref{prop33} uses in a crucial way the divisibility assumption on the $n_{k}$'s, and it comes as a natural question to ask whether it can be dispensed with:
if there exists an $a>1$ such that ${n_{k+1}}/{n_{k}}\geq a$ for any $k\geq 0$,
must $(n_{k})_{k\geq 0}$ be a \rs?
We were not able to settle this question, but it is answered in \cite{preprint} in the negative: the sequence $(n_{k})_{k\geq 0}$ with $n_{k}=2^{k}+1$ cannot be a \rs. Indeed we have $2n_{k}=n_{k+1}+1$, so that if $(n_{k})_{k\geq 0}$ were a \rs, with $\varphi $ an associated \wmx\ \mpt\ on $(X, \mathcal{F}, \mu )$, we should have both $U_{\varphi }^{2n_{k}}\to I$ (SOT) and $U_{\varphi }^{n_{k}+1}\to I$ (SOT), so that $U_{\varphi }=I$ which is impossible.
\par\smallskip
Obviously a \rs\ must have density $0$ (this is pointed out already in \cite{JKLSS}).
Some of the simplest examples of non-\rs s $(n_{k})_{k\geq 0}$ satisfy $n_{k+1}/n_{k} \to 1$. Our three Examples \ref{ex44}, \ref{ex5} and \ref{ex55} overlap with examples of \cite{preprint}.

\begin{example}\label{ex44}
 Let $p\in\Z[X]$ be a polynomial with nonnegative coefficients,
 $p\not = 0$. Then the sequence $(n_{k})_{k\geq 0}$ with $n_{k}=p(k)$ cannot be a \rs.
\end{example}

This follows directly from Weyl's polynomial equidistribution theorem (see for instance \cite[p. 27]{KN}): for any irrational number $\theta \in [0,1]$, the sequence $(p(k)\theta )_{k\geq 0}$ is uniformly equidistributed. Hence
$$\frac{1}{N}\sum_{k=1}^{N}e^{2i\pi p(k)\theta }\to 0 \quad \textrm{as } N\to +\infty $$ for every $\theta \in [0,1]\setminus \Q$. Hence if $\sigma  $ is any continuous probability measure on $\T$,
$$\frac{1}{N}\sum_{k=1}^{N} \hat{\sigma  }(n_{k})\to 0,$$ and this forbids $\hat{\sigma  }(n_{k})$ to tend to $1$. We have proved in fact:

\begin{example}\label{ex5}
 If there exists a countable subset $Q$ of $[0,1]$ and a $\delta >0$ such that for any $\theta \in [0,1]\setminus Q$,
 $$\liminf_{N\to +\infty }|\frac{1}{N}\sum_{k=1}^{N}e^{2i\pi n_{k}\theta }|\leq 1-\delta ,$$ then
 $(n_{k})_{k\geq 0}$ is not a \rs.
\end{example}

See \cite{BaG1} for some examples of such sequences. Let us point out that (contrary to what happens for Jamison sequences), it is obvious to exhibit non-rigidity sequences  $(n_{k})_{k\geq 0}$ with $\liminf n_{k+1}/n_{k}=1$ and $\limsup n_{k+1}/n_{k}=+\infty $: take any sequence  $(n_{2k})_{k\geq 0}$ such that  $n_{2k+2}/n_{2k}\to +\infty $, and set $n_{2k+1}=n_{2k}+1$. If $(n_{k})_{k\geq 0}$ were a \rs, with $\varphi $ an associated \wmx\ \mpt\ on $(X, \mathcal{F}, \mu )$, we should have $U_{\varphi }^{n_{k}}\to I$ (SOT), so that $U_{\varphi }=I$, a contradiction.
A similar type of argument yields

\begin{example}\label{ex55}
If $(n_{k})_{k\geq 0}$ denotes the sequence of prime numbers, then $(n_{k})_{k\ge 0}$ is not a \rs.
\end{example}

\begin{proof}
This follows from a result of Vinogradov that any sufficiently large odd number can be written as a sum of three primes. Suppose by contradiction that $(n_{k})_{k \geq 0}$ is a \rs\ with $\varphi $ an associated \wmx\ \mpt\ on $(X, \mathcal{F}, \mu )$. Then $U_{\varphi }^{n_{k}}\to I$ (SOT). Let $f\neq 0$ be a function in $L^{2}(X,\mathcal{F}, \mu )$ with $\int_X f d\mu =0$. If $\varepsilon >0$ is any positive number, let $k_{0}$ be such that for any $k\geq k_{0}$, $||U_{\varphi }^{n_{k}}f-f||<\varepsilon $ and every odd integer greater than or equal to $k_0$ can be written as a sum of three primes. Consider the finite set of integers $A=\{0, n_{k_{1}}, n_{k_{1}}+n_{k_{2}}, n_{k_{1}}+n_{k_{2}}+n_{k_{3}} \textrm{ ; } 0\le k_{i}\le k_{0} \textrm{ for } i=1,2,3\}$. We claim that for any odd integer $2n+1\geq k_0$, there exists an $m\in A$ such that $||U_{\varphi }^{2n+1}f-U_{\varphi }^{m}f||<3\varepsilon $. Indeed, let us write $2n+1$ as $2n+1=n_{k_{1}}+n_{k_{2}}+n_{k_{3}}$ with $0
 \le k_{1}\le k_{2}\le k_{3}$, and consider separately four cases:

 -- if $k_{1}>k_{0}$, then $||U_{\varphi }^{n_{k_{1}}+n_{k_{2}}+n_{k_{3}}}f-f||\le
 ||U_{\varphi }^{n_{k_{1}}}f-f||+||U_{\varphi }^{n_{k_{2}}}f-f||+||U_{\varphi }^{n_{k_{3}}}f-f||<3\varepsilon
 $;

 -- if $k_{1}\le k_{0}$ and $k_{2}>k_{0}$, $||U_{\varphi }^{n_{k_{1}}+n_{k_{2}}+n_{k_{3}}}f-U_{\varphi }^{n_{k_{1}}}f||\le ||U_{\varphi }^{n_{k_{2}}}f-f||+||U_{\varphi }^{n_{k_{3}}}f-f||<2\varepsilon
 $;

 -- if $k_{2}\le k_{0}$ and $k_{3}>k_{0}$, $||U_{\varphi }^{n_{k_{1}}+n_{k_{2}}+n_{k_{3}}}f-U_{\varphi }^{n_{k_{1}}+n_{k_{2}}}f||\le ||U_{\varphi }^{n_{k_{3}}}f-f||<\varepsilon
 $;

 -- if $k_{3}\le k_{0}$, there is nothing to prove.

 Now since $\varphi$ is weakly mixing, $U_{\varphi }^{2n+1}f\to 0$ (WOT) along a set $D$ which is of density $1$ in the set of odd integers. Since $A$ is finite,
it follows that there exists some $m\in A$ such that $\|U_{\varphi }^{l_j}f - U_{\varphi }^m f\|<3\eps$ for an increasing sequence $(l_j)_{j\geq 0}\subset D$. Thus, for every $g\in L^{2}(X,\mathcal{F}, \mu )$ we have
$$
|\langle U_{\varphi }^mf,g \rangle|\leq |\langle U_{\varphi }^mf - U_{\varphi }^{l_j}f,g \rangle| + |\langle U_{\varphi }^{l_j}f,g \rangle| \leq 3\eps \|g\| + |\langle U_{\varphi }^{l_j}f,g \rangle|.
$$
Taking the weak limit as $j\to \infty$ of the above expression implies $\|U_{\varphi }^mf\|\leq 3\eps$. Thus $f=0$, a contradiction.
\end{proof}

The proof of Example \ref{ex55} actually shows that if there exists an integer $r\geq 2$ such that any sufficiently large integer
in a set of positive density
can be written as a sum of $r$ elements of the set $\{n_{k} \textrm{ ; }k\geq 0\}$, then $(n_{k})_{k\geq 0} $ cannot be a \rs.
As pointed out in \cite{preprint}, the statement of Example \ref{ex55} can also be deduced from the fact that $(n_{k}x)_{k\geq 0}$ is uniformly distributed for all but a countable set of values of $x\in [0,1]$.
\par\smallskip
We finish this section with some more examples of \rs s. We consider
the sequence $(q_{n})_{n\geq 1}$ of quotients of the convergents of some irrational numbers $\alpha \in ]0,1[$. Let $\alpha $ be such  a number, and let $$\alpha =\cfrac{1}{a_{1}+\cfrac{1}{a_{2}+\cfrac{1}{a_{3}+\dots}}}
$$
with the $a_{n}$'s positive integers, be its continued fraction expansion. The convergents of $\alpha $ are the rational numbers $\frac{p_{n}}{q_{n}}$ defined recursively by the equations
\begin{equation*}
\begin{cases}
 p_{0}=0,\, p_{1}=1, \,p_{n+1}=a_{n}p_{n}+p_{n-1} \textrm{ for } n\geq 2\\
 q_{0}=1,\, q_{1}=a_1, \,q_{n+1}=a_{n}q_{n}+q_{n-1} \textrm{ for } n\geq 2.
\end{cases}
\end{equation*}
See for instance \cite{HW} for more about continued fraction expansions and approximations of irrational numbers by rationals.
We have
\begin{eqnarray}
 \label{eq10}\frac{1}{2q_{n}q_{n+1}}\leq\left|\alpha -\frac{p_{n}}{q_{n}}\right|<\frac{1}{q_{n}q_{n+1}}
\end{eqnarray}
for any $n\geq 1$. It follows that $|e^{2i\pi q_{n}\alpha }-1|\to 0$  as $n\to +\infty $. Hence there exist infinitely many numbers $\lambda \in\T\setminus\{1\}$ such that $|\lambda ^{q_{n}}-1|\to 0$ as $n\to +\infty $, and the sequence $(q_{n})_{n\geq 1}$ is a possible candidate for a \rs. We begin by recalling a particular case of a result of Katok and Stepin \cite{KS}, see also \cite{Queff}:

\begin{example}\label{ex66}
If, with the notation above, $$|\alpha -\frac{p_{n}}{q_{n}}|=o(\frac{1}{q_{n}^{2}}),$$ then $(q_{n})_{n\geq 0}$  is a \rs.
\end{example}

This can also be seen as a direct consequence of our Example \ref{ex1}: by the lower bound in (\ref{eq10}), the assumption is equivalent to ${q_{n+1}}/{q_{n}}\to +\infty $ (i.e. $a_{n}\to +\infty $). It is also possible to show that $(q_{n})_{n\geq 0}$  is a \rs\ (and even more) for some irrational numbers $\alpha $ with $\liminf a_{n}<+\infty $. For instance:

\begin{example}\label{ex2}
 Let $m\geq 2$ be an integer, and let $\alpha _{m}$ be the Liouville number $$\alpha _{m}=\sum_{k\geq 0}m^{-(k+1)!}.$$ If $(q_{n})_{n\geq 1}$ denotes the sequence of denominators of the convergents of $\alpha _{m}$, then
  there exists a perfect compact subset of $\T$ on which $\lambda ^{q_{n}}$ tends uniformly to $1$. In particular
$(q_{n})_{n\geq 1}$ is a \rs.
\end{example}

\begin{proof}
The proof relies on a paper of Shallit \cite{Sha} where the continued fraction expansion of $\alpha _{m}$  is determined: if $[a_{0}, a_{1},\ldots, a_{N_{v}}]$ is the continued fraction expansion of $\sum_{k= 0}^{v} m^{-(k+1)!}$, $v$ a nonnegative integer, then the continued fraction expansion of the next partial sum $\sum_{k= 0}^{v+1} m^{-(k+1)!}$ is given by $$[a_{0}, a_{1},\ldots, a_{N_{v+1}}]=[a_{0}, a_{1},\ldots, a_{N_{v}}, m^{v(v+1)!}-1,1,a_{N_{v}}-1,a_{N_{v}-1},\ldots, a_{2},a_{1}]$$ as soon as $N_{v}$ is even. One has $N_{v+1}=2N_{v}+2$ so that $N_{v+1}$ is indeed even. This yields that the continued fraction expansion of $\alpha _{m}$ is
$$[0,m-1,m+1,m^{2}-1,1,m,m-1,m^{12}-1,1,m-2,m,1,m^{2}-1,m+1,m-1,m^{72}-1,1,\ldots].$$ We have $a_{N_{v}+1}=m^{(v-1)v!}-1$.
For any $v\geq 0$, $$\frac{q_{N_{v}}+2}{q_{N_{v}}+1}=m^{(v-1)v!}-1+\frac{q_{N_{v}}}{q_{N_{v}}+1}\geq m^{(v-1)v!}-1
\geq \frac{1}{2}m^{(v-1)v!}\quad \textrm{for}\quad  v\geq 2.$$
Applying the proof of Proposition \ref{prop3} to the sequence $(n_{v})_{v\geq 0}=(q_{N_{v}+1})_{v\geq 0}$, we get that there exists a  perfect compact subset $K$ of $\T$ containing the point $1$ such that for any $\lambda \in K$ and any $v\geq 1$,
$$|\lambda ^{q_{N_{v}+1}}-1|\leq 10\pi\,\sup_{j\geq v}\frac{q_{N_{j-1}+1}}{q_{N_{j}+1}}\leq
10\pi\, \frac{q_{N_{v}+1}}{q_{N_{v+1}+1}}\leq 10\pi\,\frac{q_{N_{v}+2}}{q_{N_{v}+1}}\leq 20\pi\,
m^{-(v-1)v!}. $$ Let now $p$ be an integer such that $N_{v-1}+2\leq p\leq N_{v}$ for some $v\geq 0$, and $\lambda \in K$. We need to estimate $|\lambda ^{p}-1|$. If $p=N_{v-1}+2$, we have $q_{N_{v-1}+2}=a_{N_{v-1}+2}\,q_{N_{v-1}+1}+q_{N_{v-1}}\leq (a_{N_{v-1}+2}+1)q_{N_{v-1}+1}$. In the same way
$q_{N_{v-1}+3}\leq (a_{N_{v-1}+2}+1)(a_{N_{v-1}+3}+1)q_{N_{v-1}+1}$
etc., and
$$q_{N_{v-1}+j}\leq\prod_{i=2}^{j}(a_{N_{v-1}+i}+1) \,q_{N_{v-1}+1}\quad \textrm{ for any } 2\leq j\leq N_{v}-N_{v-1}=N_{v-1}+2. $$ So for $N_{v-1}+2\leq p\leq N_{v}$ we have
$$q_{p}\leq\prod_{i=2}^{p-N_{v-1}}(a_{N_{v-1}+i}+1) \,q_{N_{v-1}+1}$$ so that
$$|\lambda ^{q_{p}}-1|\leq \prod_{i=2}^{N_{v-1}+2}(a_{N_{v-1}+i}+1)\, |\lambda ^{q_{N_{v-1}+1}}-1| .$$ It remains to estimate the quantity $\prod_{i=2}^{N_{v-1}+2}(a_{N_{v-1}+i}+1)$. We have
$\{a_{N_{v-1}+2}, \ldots, a_{N_{v}}\}=\{1, a_{N_{v-1}}-1, a_{N_{v-1}-1}, \ldots, a_{2},a_{1}\}$ so that
$\prod_{i=2}^{N_{v-1}+2}(a_{N_{v-1}+i}+1)\leq 2 \prod_{i=2}^{N_{v-1}}(a_{i}+1)$. Let us write
$R_{v-1}=\prod_{i=2}^{N_{v-1}}(a_{i}+1)$. We have $R_{v}\leq R_{v-1}\, m^{(v-2)(v-1)!}\, 2\, R_{v-1}$ by the inequality above, i.e.
\begin{eqnarray*}
R_{v}\leq 2\, R_{v-1}^{2}m^{(v-2)(v-1)!} &\leq& 2^{1+2} R_{v-2}^{4} m^{(v-2)(v-1)!+2(v-3)(v-2)!}\\
&\leq& \ldots\leq 2^{2^{v+1}}
m^{\sum_{k=1}^{v-1}2^{k-1}(v-(k+1))(v-k)!}.
\end{eqnarray*}
Now $(v-1)!\geq 2^{k-1}(v-k)!$, so
$$R_{v}\leq 2^{2^{v+1}}m^{(v-1)(v-2)(v-1)!}.$$ Hence
$$|\lambda ^{q_{p}}-1|\leq 2^{2^{v+1}}m^{(v-1)(v-2)(v-1)!} |\lambda ^{q_{N_{v-1}+1}}-1| $$
for any $N_{v-1}+2\leq p\leq N_{v}$. Now we have $$|\lambda ^{q_{N_{v-1}+1}}-1|\leq 20\pi\, m^{-(v-1)v!}$$ so that
$$|\lambda ^{q_{p}}-1|\leq 20\pi\, 2^{2^{v+2}}m^{(v-1)(v-1)! (v-2-v)} =20\pi\, 2^{2^{v+2}}m^{-2(v-1)(v-1)!}$$ for $N_{v-1}+2\leq p\leq N_{v}$.
 Since the quantity $2^{2^{v+2}}m^{-2(v-1)(v-1)!}$ tends to $0$ as $v$ tends to infinity, it follows that $\lambda ^{q_{n}}$ tends to $1$ uniformly on $K$ as $n$ tends to infinity.
\end{proof}

A stronger result is proved in \cite{preprint}: actually if $\alpha $ is any irrational number in $]0,1[$, the sequence $(q_{n})_{n\geq 0}$ of denominators of the convergents of $\alpha $ is always a \rs.
\par\smallskip
We finish our study of \rs s by giving an example of a \rs\ such that $n_{k+1}/n_{k}\to 1$. This answers a question of \cite{preprint}.

\begin{example}\label{ex77}
There exists a sequence $(n_{k})_{k\geq 0}$ with $n_{k+1}/n_{k}\to 1$ as $k\to +\infty $ which is a \rs.
\end{example}

\begin{proof}
Let $(k_{p})_{p\geq 2}$ be a very quickly increasing sequence of integers with $k_1=1$ which will be determined later on in the proof. For $p\geq 0$, let $N_{p}=2^{2^{p}}$, and consider the set $$A_{N_{p}}=\bigcup_{k=k_{p+1}}^{2k_{p+2}-1}A_{N_p,k}$$ where
$$A_{N_{p},k}=\{N_{p}^{k}, N_{p}^{k}+N_{p}^{k-1}, N_{p}^{k}+2N_{p}^{k-1}, N_{p}^{k}+3N_{p}^{k-1},
\ldots, N_{p}^{k}+((N_{p}-1)N_{p}-1)N_{p}^{k-1}\}.$$
For instance, $$A_{2}=\bigcup_{k={1}}^{2k_{2}-1}\{2^{k},2^{k}+2^{k-1}\},\;
\;
A_{4}=\bigcup_{k=k_{2}}^{2k_{3}-1}\{4^{k},4^{k}+4^{k-1}, 4^{k}+2 \,4^{k-1}, \ldots, 4^{k}+11\, 4^{k-1}\}, \textrm{ etc.}$$
As the last element of $A_{N_{p}}$ is $N_{p}^{2k_{p+2}}-N_{p}^{2k_{p+2}-2}$ which is less than the first element of
$A_{N_{p+1}}$, $N_{p+1}^{k_{p+2}}=N_{p}^{2k_{p+2}}$, these sets are successive and disjoint. Let $(n_{j})_{j\geq 0}$ be the strictly increasing sequence such that $A=\bigcup_{p\geq 0}A_{N_{p}}=\{n_{j} \textrm{ ; } j\geq 0\}$.
Let us first check that $n_{j+1}/n_{j} \to 1$:
first of all, if $n_{j}$ and $n_{j+1}$ belong to the same set $A_{N_{p},k}$, $$\frac{n_{j+1}}{n_{j}}=\frac{N_{p}^{k}+l\,N_{p}^{k-1}}{N_{p}^{k}+(l-1)\,N_{p}^{k-1}}=1+\frac{N_{p}^{k-1}}{N_{p}^{k}+(l-1)\,N_{p}^{k-1}}\leq 1+\frac{1}{N_{p}}\cdot$$ If $n_{j}$ is in some set $A_{N_{p},k}$ and $n_{j+1}$ is in $A_{N_{p},k+1}$,
$$\frac{n_{j+1}}{n_{j}}=\frac{N_{p}^{k+1}}{N_{p}^{k+1}-N_{p}^{k-1}}=\frac{1}{1-\frac{1}{N_{p}^{2}}}=\frac{N_{p+1}}{N_{p+1}-1}\cdot$$ Lastly, if $n_{j}$ is the last integer of $A_{N_{p}}$ and $n_{j+1}$ is the first integer of $A_{N_{p+1}}$,
$$\frac{n_{j+1}}{n_{j}}=\frac{N_{p+1}^{k_{p+2}}}{N_{p}^{2k_{p+2}}-N_{p}^{2k_{p+2}-2}}=
\frac{N_{p}^{2k_{p+2}}}{N_{p}^{2k_{p+2}}-N_{p}^{2k_{p+2}-2}}=\frac{N_{p+1}}{N_{p+1}-1}\cdot$$ Thus $n_{j+1}/n_{j} \to 1$.
Let now  $\sigma  $ be a continuous probability \mea\ on $\T$ such that
\begin{itemize}
 \item for any $0\le k \le 2k_{2}-1$, $|\hat{\sigma  }(2^{k})-1|\leq a_{k}$

 \item for any $p\geq 1$  and $k_{p+1}\le k \le 2k_{p+2}-1$, $$|\hat{\sigma  }(N_{p}^{k})-1|\leq \frac{a_{2k_{p+1}-1}}{a_{k_{p+1}}} a_{k}$$
\end{itemize}
where $a_{0}=a_{1}=1$ and $a_{k}=\frac{1}{k\log k}$ for $k\geq 2$.
Such a \mea\ does exist by Proposition \ref{prop33} and Remark \ref{rem33}. Indeed the successive terms of the sequence $$(1,2,4,\ldots, 2^{2k_{2}-1}, 4^{k_{2}}, 4^{k_{2}+1}, \ldots )=(m_{j})_{j\geq 0}$$ divide each other. The sequence $(a_{0},a_{1}, \ldots, a_{2k_{2}-1},\frac{a_{2k_{2}-1}}{a_{k_{2}}} a_{k_{2}}, \frac{a_{2k_{2}-1}}{a_{k_{2}}} a_{k_{2}+1}, \ldots)=(b_{j})_{j\geq 0}$ is decreasing to zero, and $\sum b_{j}$ is divergent: if the sequence $(k_{p})$ grows fast enough,
$$\sum_{j\geq 0} b_{j}\geq \sum_{k=2}^{2k_{2}-1}a_{k}+ \frac{a_{2k_{2}-1}}{a_{k_{2}}} \sum_{k=k_{2}}^{2k_{3}-1}a_{k}+\ldots$$ and since the series $\sum a_{k}$ is divergent, it is possible to choose $k_{p+1}$ so large \wrt\ $k_p$ that $$\frac{a_{2k_{p}-1}}{a_{k_{p}}}
\sum_{k=k_{p}}^{2k_{p+1}-1}a_{k}\geq 1$$ for instance for each $p$. So we have a \proba\ \mea\ $\sigma  $
on $\T$ such that $|\hat{\sigma  }({m_{j}})-1|\leq b_{j}$ for each $j\ge 0$.
It remains to show that $|\hat{\sigma  }({n_{k}})-1|\to 0$. For $k_{p+1}\le k\le 2k_{p+2}-1$ and $0\le l\le (N_{p}-1)N_{p}-1$, we have
\begin{eqnarray*}
|\hat{\sigma  }(N_{p}^{k}+l\,N_{p}^{k-1})-1|&\leq& |\hat{\sigma  }(N_{p}^{k})-1|+l\,|\hat{\sigma  }(N_{p}^{k-1})-1|\\
&\leq& |\hat{\sigma  }(N_{p}^{k})-1|+ ((N_{p}-1)N_{p}-1)
|\hat{\sigma  }(N_{p}^{k-1})-1|.
\end{eqnarray*}
If $k_{p+1}+1\le k \le 2k_{p+2}-1$, this is less than
\begin{eqnarray*}
\frac{a_{2k_{p+1}-1}}{a_{k_{p+1}}}(a_{k}+((N_{p}-1)N_{p}-1)a_{k-1})&\leq& (N_{p}-1)N_{p}\,
\frac{a_{2k_{p+1}-1}}{a_{k_{p+1}}}\,a_{k_{p+1}}\\
&\le& (N_{p}-1)N_{p}\, a_{2k_{p+1}-1}.
\end{eqnarray*}
If $k_{p+1}$ is large enough compared to $N_{p}$, this quantity is less than $2^{-p}$. If $k=k_{p+1}$, $N_{p}^{k_{p+1}-1}=N_{p-1}^{2k_{p+1}-2}$. Hence
\begin{eqnarray*}
|\hat{\sigma  }(N_{p}^{k_{p+1}}+l\,N_{p}^{k_{p+1}-1})-1|&\leq& a_{2k_{p+1}-1}+((N_{p}-1)N_{p}-1)\,\frac{a_{2k_{p}-2}}{a_{k_{p}}}\, a_{2k_{p+1}-2}\\
&\leq& (1+((N_{p}-1)N_{p}-1)\,\frac{a_{2k_{p}-2}}{a_{k_{p}}}) a_{2k_{p+1}-2}
\end{eqnarray*}
and this again can be made less than $2^{-p}$ provided $k_{p+1}$ is sufficiently large \wrt\ $N_{p}$ and $k_{p}$.
Hence $\hat{\sigma  }(n_{k})\to 1$ as $k\to +\infty $, and this proves that $(n_{k})_{k\geq 0}$ is a \rs.
\end{proof}

\section{Topologically and uniformly rigid linear dynamical systems}

\subsection{Back to rigidity in the linear framework}
Before moving over to topological versions of rigidity for linear dynamical systems, we have to settle the following natural question: which sequences $(n_{k})_{k\geq 0}$ appear as rigidity sequences (in the \mea-theoretic sense) for linear dynamical systems? Here is the answer:

\begin{theorem}\label{th5}
 Let $(n_{k})_{k\geq 0}$ be an increasing sequence of positive integers. The following assertions are equivalent:
 \begin{itemize}
  \item[(1)] there exists a continuous \proba\ \mea\ $\sigma  $  on
    $\T$  such that $\hat{\sigma  }(n_{k})\to 1$ as
    $k\to +\infty $,
i.e., $(n_{k})_{k\geq 0}$ is a rigidity sequence;

\item[(2)] there exists a bounded linear \op\ $T$ on a \sep\ complex
  infinite-dimensional Hilbert space $H$ which admits a \nd\ \ga\
  \mea\ $m$ \wrt\ which $T$ defines a \wmx\ \mpt\ which is \rg\ \wrt\
  $(n_{k})_{k\geq 0}$.
 \end{itemize}
\end{theorem}

\begin{proof}
 The implication $(2)\Rightarrow (1)$ is an immediate consequence of Proposition \ref{prop4}, so let us prove that $(1)\Rightarrow (2)$.
 Let $\sigma  $ be a continuous \proba\ \mea\ $\sigma  $  on $\T$
 such that $\hat{\sigma  }(n_{k})\to 1$ as $n_{k}\to
 +\infty $, and let $L\subseteq \T$ be the support of the measure
 $\sigma  $. It is a compact perfect subset of $\T$, and $\sigma
 (\Omega )>0$ for any non-empty open subset $\Omega$ of $L$. Kalish constructed in \cite{Ka} an example of a bounded \op\ on a Hilbert space whose point spectrum is equal to $L$, and, as in \cite{BG}, we use this example for our purposes: let $T_{0}$ be the \op\ defined on $L^{2}(\T)$ by $T_{0}=M-J$, where $Mf(\lambda )=\lambda f(\lambda )$ and $Jf(\lambda )=\int_{(1,\lambda )}f(\zeta )d\zeta $ for any $f\in L^{2}(\T)$ and $\lambda \in\T$.
 For $\lambda \in\T$, $\lambda =e^{i\theta }$, $(1,\lambda )$ denotes the arc $\{e^{i\alpha } \textrm{ ; } 0\leq \alpha \leq \theta  \}$, and $(\lambda ,1)$ the arc $\{e^{i\alpha } \textrm{ ; } \theta \leq \alpha \leq 2\pi \}$. For every $\lambda $, the characteristic function $\chi_{\lambda }$ of the arc $(\lambda ,1)$ is an \eve\ of $T_{0}$ associated to the \eva\ $\lambda $. Let $T$ be the \op\ induced by $T_{0}$ on the space $H=\overline{\textrm{sp}}[\chi_{\lambda } \textrm{ ; } \lambda \in L]$. It is proved in \cite{Ka} that $\sigma  (T)=\sigma  _{p}(T)=L$, and it is not difficult to see that $E:\lambda \mapsto \chi_{\lambda }$ is a continuous \eve\ field for $T$ on $L$ which is spanning. Hence it is a perfectly spanning unimodular \eve\ field \wrt\ the \mea\ $\sigma  $ (see \cite{BG} for details), and there exists a \nd\ \ga\ \mea\ $m$ on $H$ whose covariance \op\ $S$ is given by $S=KK^{*}$, where
 $K:L^{2}(\T,\sigma  )\to H$ is the \op\ defined by $K\varphi=\int_{\T}\varphi(\lambda )E(\lambda )d\sigma  (\lambda )$ for $\varphi \in L^{2}(\T,\sigma  )$,
 \wrt\ which $T$ defines a \wmx\ \mpt. It remains to prove that $T$ is \rg\ \wrt\ $(n_{k})_{k\geq 0}$, i.e. that $U_{T}^{n_{k}}f$ tends weakly to $f$ in
 $L^{2}(H, \mathcal{B},m)$. Using the same kind of arguments as in \cite{BG2} or \cite[Ch. 5]{BM}, we see that it suffices to prove that for any elements $x,y$ of $H$,
 $$\int_{H}\pss{x}{T^{n_{k}}z}\overline{\pss{y}{z}}dm(z)\longrightarrow \int_{H}\pss{x}{z}\overline{\pss{y}{z}}dm(z)
 \qquad \textrm{as } n_{k}\to +\infty .$$ But this is clear:
since
 $TK=KV$, where $V$ is the multiplication \op\ by $\lambda $ on $L^{2}(\T,\sigma  )$, we have
 \begin{eqnarray*}
   \int_{H}\pss{x}{T^{n_{k}}z}\overline{\pss{y}{z}}dm(z)&=&\pss{KK^{*}T^{*n_{k}}x}{y}=\pss{V^{*n_{k}}K^{*}x}{K^{*}y}\\
   &=&\int_{\T}\lambda ^{-n_{k}}\pss{x}{E(\lambda )}\overline{\pss{y}{E(\lambda )}}d\sigma  (\lambda ).
 \end{eqnarray*}
The function $h(\lambda )=\pss{x}{E(\lambda )}\overline{\pss{y}{E(\lambda )}}$ belongs to $L^{1}(\T,\sigma  )$,  and we have seen in the proof of Proposition \ref{prop4} that
 $\int_{\T}|\lambda ^{n_{k}}-1|\, |h(\lambda )| d\sigma  (\lambda )\to 0$. Hence
 $$\int_{\T}\lambda ^{-n_{k}}
h(\lambda)
d\sigma  (\lambda )\to \int_{\T}h(\lambda )d\sigma  (\lambda )=
 \int_{H}\pss{x}{z}\overline{\pss{y}{z}}dm(z),$$ and this proves our statement.
 \end{proof}

\begin{remark}
The Kalish-type \ops\ which are used in the proof of Theorem \ref{th5} have no reason at all to be power-bounded \wrt\ $(n_{k})$, contrary to what happens when considering topological rigidity. We only know for instance, applying the rigidity assumption to the function $f(z)=||z||$, that
$$\int_{H}||(T^{n_{k}}-I)z||\, dm(z)\to 0 \qquad \textrm{as } n_{k}\to +\infty .$$
\end{remark}

\begin{remark}
Theorem \ref{th5} gives another proof of the characterization of rigidity sequences obtained in Proposition \ref{prop4}.
\end{remark}

\subsection{A characterization of topologically rigid sequences for linear dynamical systems}
Let us prove Theorem \ref{th2}.
First of all, (3) implies (2) since, as recalled in Section 2.1, (3)
implies that $T$ has perfectly spanning unimodular \eve s. We suppose
next that (2) holds and show (1). Let $X$ and $T$ be as in (2). For any $\lambda \in\sigma _{p}(T)\cap\T$, let $e_{\lambda }$ be
an associated \eve\ with $||e_{\lambda }||=1$. Since $T^{n_{k}}e_{\lambda }\to e_{\lambda }$, $|\lambda ^{n_{k}}-1|\to 0$ for any $\lambda \in\sigma _{p}(T)\cap\T$. Moreover, by the uniform boundedness principle, $\sup_{k\geq 0}||T^{n_{k}}||=M$ is finite. Suppose by contradiction that there exists an $\varepsilon _{0}>0$ such that for any $\lambda,\mu  \in\sigma _{p}(T)\cap\T$ with $\lambda \not =\mu $, $\sup_{k\geq 0}|\lambda ^{n_{k}}-\mu ^{n_{k}}|\geq \varepsilon _{0}$. Then for any $\lambda,\mu  \in\sigma _{p}(T)\cap\T$,
$$|\lambda ^{n_{k}}-\mu ^{n_{k}}|-||e_{\lambda }-e_{\mu
}||\leq||\lambda ^{n_{k}}e_{\lambda }-\mu ^{n_{k}}e_{\mu }||\leq
M\,||e_{\lambda }-e_{\mu }||$$ so that $|\lambda ^{n_{k}}-\mu
^{n_{k}}|\leq (M+1)||e_{\lambda }-e_{\mu }||$. Hence $\varepsilon
_{0}\leq  (M+1)||e_{\lambda }-e_{\mu }||$, and the unimodular \eve s
of $T$ are
$\varepsilon _{0}/(M+1)$-separated. Since $X$ is \sep\ there can only
be countably many such \eve s, which contradicts the fact that
$\sigma_p(T)\cap\T$ is uncountable. So for any $\varepsilon >0$ there exist $\lambda,\mu $ in $\sigma  _{p}(T)\cap\T$ with $\lambda \not =\mu $ such that
$\sup_{k\geq 0} |(\lambda \overline{\mu})^{n_{k}}-1|\leq \varepsilon
$, and $|(\lambda \overline{\mu})^{n_{k}}-1|\to 0$. So (1)
holds true.

We state again what we have to prove in order to obtain
that (1) implies (3):

\begin{theorem}\label{th2ter}
Let $(n_{k})_{k\geq 0}$  be an increasing sequence of integers with $n_{0}=1$ such that for any $\varepsilon >0$ there exists a $\lambda \in\T\setminus\{1\}$ with $$
\sup_{k\geq 0}|\lambda ^{n_{k}}-1|\leq\varepsilon \quad \textrm{ and }\quad |\lambda ^{n_{k}}-1|\to 0
\quad \textrm{ as }\quad
k\to\infty.
$$ Then there exists a bounded linear \op\ $T$ on a Hilbert space $H$
such that $T$ has a \ps\ and for every $x\in H$, $T^{n_{k}}x
\to x$ as
$k\to\infty$.
\end{theorem}

Before starting the proof of Theorem \ref{th2ter}, let us point out that the statement is not true anymore if we only suppose that there exists a $\lambda \in\T\setminus\{1\}$ such that $|\lambda ^{n_{k}}-1|\to 0$: if $(q_n)_{n\geq 0}$ is the sequence of denominators of the partial quotients in the continued fraction expansion of $\alpha=\sqrt{2}$ for instance, $\lambda=e^{2i\pi\alpha}$ is such that $|\lambda ^{q_{n}}-1|\to 0$. But the sequence $(\frac{q_{n+1}}{q_n})_{n\geq 0}$ is bounded (see for instance \cite{HW}), so that $(q_n)_{n\geq 0}$ is not even a \js.

\begin{proof}[Proof of Theorem \ref{th2ter}]
We take the same kind of \op\ as in the proof of Theorem \ref{th1bis}, and show that under the assumptions of Theorem \ref{th2}, such an \op\
$T=D+B$ is such that $||T^{n_{p}}-D^{n_{p}}||$ tends to $0 $ as $n_{p}$ tends to infinity.
Before starting on this, we take advantage of the assumption of the theorem to construct a particular perfect compact subset of $\T$, in which our coefficients
$\lambda _{l}$ will be chosen later in the proof:

\begin{lemma}\label{lem5}
Under the assumptions of Theorem \ref{th2ter}, there exists a perfect compact subset $K$ of $\T$ such that $(K,d_{(n_{k})})$ is \sep\ and for any $\lambda \in K$, $|\lambda ^{n_{k}}-1|\to 0$ as $n_{k}\to +\infty $.
\end{lemma}

\begin{proof}[Proof of Lemma \ref{lem5}]
The proof proceeds along the same lines as in \cite{BaG2}:
let $(\mu _{n})_{n\geq 1}$
 be a sequence of elements of $\T\setminus\{1\}$ such that
 $$d_{(n_{k})}(\mu _{1},1)<4^{-1}
 \textrm{ , } d_{(n_{k})}(\mu _{n},1)<4^{-n}d_{(n_{k})}
 (\mu _{n-1}, \overline{\mu
 }_{n-1}) \textrm{ for any } n\geq 2,$$
 $d_{(n_{k})}(\mu _{n},\overline{\mu }_{n})$ decreases
 with $n$, and moreover $|\mu_{n}^{n_{k}}-1|\to 0$ as $n_{k}\to +\infty $.
If $(s_{1}, \ldots, s_{n})$
is any finite sequence of zeroes and ones, we associate to it an element $\lambda _{(s_{1}, \ldots, s_{n})}$ of $\T$ in the
following way: we start with $\lambda _{(0)}=\mu _{1}$ and $\lambda _{(1)}=\overline{\mu
}_{1}$, and we have
$$d_{(n_{k})}(\lambda _{(0)},\lambda _{(1)})=
d_{(n_{k})}(\mu _{1},\overline{\mu _{1}})>0.$$
Then if $\lambda _{(s_{1},\ldots,s_{n-1})}$
 has already been defined, we set $$\lambda _{(s_{1},\ldots,s_{n-1},0)}=
 \lambda _{(s_{1},\ldots,s_{n-1})} \mu _{n}\quad   \textrm{and}\quad
 \lambda _{(s_{1},\ldots,s_{n-1},1)}=\lambda _{(s_{1},\ldots,s_{n-1})}\overline{\mu }_{n}.$$ We have
\begin{equation*}
d_{(n_{k})}(\lambda _{(s_{1},\ldots,s_{n-1})},
\lambda _{(s_{1},\ldots,s_{n-1},s_{n})} )
<4^{-n} d_{(n_{k})}(\mu _{n-1},\overline{\mu }_{n-1})
\end{equation*}
and
\begin{equation*}
d_{(n_{k})}(\lambda _{(s_{1},\ldots,s_{n-1},0)},
\lambda _{(s_{1},\ldots,s_{n-1},1)} )=d_{(n_{k})}(\mu _{n},\overline{\mu }_{n}),
\end{equation*}
so that for any infinite sequence $s=(s_{1},  s_{2}, \ldots)$ of zeroes and ones,
we can define $\lambda _{s}\in\T$ as $\lambda_{s}=
\lim_{n\to +\infty }\lambda _{(s_{1},\ldots, s_{n})}.$
It is not difficult to check (see \cite{BaG2} for details) that the map $s\longmapsto \lambda _{s}$ from $2^{\omega }$ into $\T$ is
one-to-one, so that $K=\{ \lambda _{s} \textrm{ ; } s\in 2^{\omega }\}$ is homeomorphic to the Cantor set, hence compact and perfect, and that $(K,d_{(n_{k})})$ is \sep. It remains to see that for any $s\in 2^{\omega }$, $|\lambda _{s}^{n_{k}}-1|\to 0$ as $n_{k}\to +\infty $. We have for any $p\geq 1$
\begin{equation*}
\lambda _{s}=\lambda _{(s_{1},\ldots,s_{p})} \prod_{j\geq p}
\lambda _{(s_{1},\ldots,s_{j+1})} \overline{\lambda
}_{(s_{1},\ldots,s_{j})},
\end{equation*}
so that for any $p\geq 1$,
\begin{eqnarray*}
|\lambda _{s}^{n_{k}}-1|
&=&\left|
\lambda _{(s_{1},\ldots,s_{p})}^{n_{k}} \prod_{j\geq p}
\lambda _{(s_{1},\ldots,s_{j+1})}^{n_{k}} \overline{\lambda
}_{(s_{1},\ldots,s_{j})}^{n_{k}}
-1\right|\\
&\leq& |\lambda _{(s_{1},\ldots,s_{p})}^{n_{k}}-
1|+
\left|\prod_{j\geq p}
\lambda _{(s_{1},\ldots,s_{j+1})}^{n_{k}}-{\lambda
}_{(s_{1},\ldots,s_{j})}^{n_{k}}\right|.
\end{eqnarray*}
Hence
\begin{eqnarray*}
|\lambda _{s}^{n_{k}}-1|&\leq & |\lambda _{(s_{1},\ldots,s_{p})}^{n_{k}}-
1|
+ \sum_{j\geq p} d_{(n_{k})}(\lambda _{(s_{1},\ldots,s_{j+1})},\lambda _{(s_{1},\ldots,s_{j})})\\
&\leq&|\lambda _{(s_{1},\ldots,s_{p})}^{n_{k}}-
1|+2\sum_{j\geq p}4^{-(j+1)}d_{(n_{k})}(\mu _{j},\overline{\mu}_{j})\\
&\leq& |\lambda _{(s_{1},\ldots,s_{p})}^{n_{k}}-
1|+2\,d_{(n_{k})}(\mu _{p},\overline{\mu}_{p})\sum_{j\geq p}4^{-(j+1)}\\
&=&|\lambda _{(s_{1},\ldots,s_{p})}^{n_{k}}-
1|+\frac{2}{3}\,4^{-p}d_{(n_{k})}(\mu _{p},\overline{\mu}_{p}).
\end{eqnarray*}
Given any $\gamma >0$, take $p$ such that the second term is less than $\gamma /2$.
Since $|\mu _{n}^{n_{k}}-1|\to 0$ as $n_{k}\to +\infty $, $|\lambda _{(s_{1},\ldots,s_{p})}^{n_{k}}-
1|\to 0$ as $n_{k}\to +\infty $ for any finite sequence $(s_{1},\ldots, s_{p})$. Hence there exists an integer $k_{0} \geq 1$ such that for any $k\geq k_{0}$,
$|\lambda _{(s_{1},\ldots,s_{p})}^{n_{k}}-1|\leq \gamma /2$. Thus for any $k\geq k_{0}$ and any $s\in 2^{\omega }$, we have $|\lambda _{s}^{n_{k}}-1|<\gamma $. So we have proved that for any $\lambda \in K$, $|\lambda ^{n_{k}}-1|\to 0$ as $n_{k}\to +\infty $.
\end{proof}
Let us now go back to the proof of Theorem \ref{th2}.
We have seen that $$||T^{n_{p}}-D^{n_{p}}||^{2}\le \sum_{l\geq 2}\sum_{k=\max(1, l-n_{p})}^{l-1}|t_{k,l}^{(n_{p})}|^{2},$$ and that it is possible for each $l\geq 2$ to take $\lambda _{l}$ with $d_{(n_{p})}({\lambda _{l},\lambda _{j(l)}})$ so small that
$$\sum_{k=\max(1, l-n_{p})}^{l-1}|t_{k,l}^{(n_{p})}|^{2}\le 2^{-l}.$$
So we do the construction in this way with the additional requirement that for each $l\geq 1$, $\lambda _{l}$ is such that $|\lambda _{l}^{n_{p}}-1|\to 0$ as $n_{p}\to+\infty $ (this is possible by Lemma \ref{lem5}).
Let now $\varepsilon >0$ and $l_{0}\geq 2$ be such that $\sum_{l\geq l_{0}+1}2^{-l}<\frac{\varepsilon }{2}$. We have for any $p$ such that $n_{p}\geq l_{0}+1$
$$||T^{n_{p}}-D^{n_{p}}||^{2}\le\sum_{l=2}^{l_{0}}\sum_{k=1}^{l-1}|t_{k,l}^{(n_{p})}|^{2}+\frac{\varepsilon }{2}\cdot$$
The proof will be complete if we show that for any $k,l$ with $1\le k\le l-1$, $t_{k,l}^{(n_{p})}\to 0$ as $n_{p}\to+\infty $, or, equivalently, that
$s_{k,l}^{(n_{p})}\to 0$. Recall that by Lemma \ref{lem3}, $s_{k,l}^{(n_{p})}$ can be written as
$$s_{k,l}^{(n_{p})}=\sum_{j=k}^{l-1}c_{j}^{(k,l)}\,\frac{\lambda _{l}^{n_{p}+1-(l-k)}-\lambda _{j}^{n_{p}+1-(l-k)}}{\lambda _{l}-\lambda _{j}}$$ as soon as $n_{p}\geq l-k$. Since $\lambda _{j}^{n_{p}}\to 1$ for any $j\geq 1$,
$$s_{k,l}^{(n_{p})}\to s_{k,l}:=\sum_{j=k}^{l-1}c_{j}^{(k,l)}\,\frac{\lambda _{l}^{1-(l-k)}-\lambda _{j}^{1-(l-k)}}{\lambda _{l}-\lambda _{j}}\quad \textrm{as }n_{p}\to+\infty .$$ Thus we have to show that $s_{k,l}=0$ for any $1\le k\le l-1$. This is a consequence of the following lemma:

\begin{lemma}\label{lem6}
 For any $k,l$ with $1\le k\le l-1$ and any $p$ with $0\le p\le l-k-1$, we have
 $$\sum_{j=k}^{l-1}c_{j}^{(k,l)}\,\frac{\lambda _{l}^{1-(l-k-p)}-\lambda _{j}^{1-(l-k-p)}}{\lambda _{l}-\lambda _{j}}=0.$$
\end{lemma}

\begin{proof}
The proof is done by induction on $l\geq 2$. For $l=2$, we just have to check that $$c_{1}^{(1,2)}\, \frac{\lambda _{2}^{0}-\lambda _{1}^{0}}{\lambda _{2}-\lambda _{1}}=0,$$ which is obviously true. Supposing now that the induction assumption is true for some $l\geq 2$, consider $k$ with $1\le k\le l$ and $p$ with $0\le p\le l-k$. Then
$$\sum_{j=k}^{l}c_{j}^{(k,l+1)}\,\frac{\lambda _{l+1}^{-(l-k-p)}-\lambda _{j}^{-(l-k-p)}}{\lambda _{l+1}-\lambda _{j}}$$
is equal to
\begin{eqnarray*}
&& -{\lambda _{l+1}^{-(l-k-p)}}
 \sum_{j=k}^{l}{c_{j}^{(k,l+1)}}{\lambda _{j}^{-(l-k-p)}}\,\frac{\lambda _{l+1}^{(l-k-p)}-\lambda _{j}^{(l-k-p)}}{\lambda _{l+1}-\lambda _{j}}\\
&& \quad = -{\lambda _{l+1}^{-(l-k-p)}}
 \sum_{j=k}^{l}{c_{j}^{(k,l+1)}}{\lambda _{j}^{-(l-k-p)}}\,\sum_{m=0}^{l-k-p-1}\lambda _{j}^{m}\lambda _{l+1}^{l-k-p-1-m}\\
&& \quad = -\lambda _{l+1}^{-1} \sum_{m=0}^{l-k-p-1} \lambda _{l+1}^{-m}\left(\sum_{j=k}^{l}{c_{j}^{(k,l+1)}}{\lambda _{j}^{-(l-k-p-m)}}\right)\cdot
\end{eqnarray*}
It suffices now to show that each sum $$\sum_{j=k}^{l}{c_{j}^{(k,l+1)}}{\lambda _{j}^{-(l-k-p-m)}}$$ is equal to $0$. We have seen in the proof of Lemma \ref{lem3} that for $1\le k\le l-1$,
$$c_{j}^{(k,l+1)}=-\frac{\lambda _{j}c_{j}^{(k,l)}}{\lambda _{l}-\lambda _{j}}\quad \textrm{ for } k\le j\le l-1 \quad \textrm{ and } c_{l}^{(k,l+1)}=\sum_{j=k}^{l-1}\frac{\lambda _{l}c_{j}^{(k,l)}}{\lambda _{l}-\lambda _{j}}$$ and that $c_{l}^{(l,l+1)}=1$. Thus for $1\le k\le l-1$
\begin{eqnarray*}
 \sum_{j=k}^{l} {c_{j}^{(k,l+1)}}{\lambda _{j}^{-(l-k-p-m)}}&=&
 -\sum_{j=k}^{l-1} \frac{c_{j}^{(k,l)}}{\lambda _{l}-\lambda _{j}} \lambda _{j}^{-(l-k-p-m-1)}
 +\sum_{j=k}^{l-1} \frac{c_{j}^{(k,l)}}{\lambda _{l}-\lambda _{j}}\lambda _{l}^{-(l-k-p-m-1)}\\
 &=&\sum_{j=k}^{l-1}c_{j}^{(k,l)}\frac{\lambda _{l}^{1-(l-k-p-m)}-\lambda _{j}^{1-(l-k-p-m)}}{\lambda _{l}-\lambda _{j}}\cdot
\end{eqnarray*}
Since $p+m\le l-k-1$, this quantity vanishes by the induction assumption. For $k=l$, we only have to consider the case $p=0$, and here
$$c_{l,l+1}^{(l)}\,\frac{\lambda _{l+1}^{0}-\lambda _{l}^{0}}{\lambda _{l+1}-\lambda _{l}}=0.$$ This finishes the proof of Lemma \ref{lem6}.
\end{proof}
We have shown that $t_{k,l}^{(n_{p})}\to 0$ for each $1\le k\le l-1$, and it follows immediately that $||T^{n_{p}}-D^{n_{p}}||\to 0$ as $n_{p}\to +\infty $.
Now if $x\in H$ and $\varepsilon $
is any fixed positive number, take $l_{0}$ such that $\sum_{l\geq l_{0}+1}|x_{l}|^{2}\leq \varepsilon /2$.
Since
$$
||D^{n_{p}}x-x||^{2} \leq \left(\sum_{l=1}^{l_{0}}|\lambda _{l}^{n_{p}}-1|^{2}\right) \,||x||^{2}+2\sum_{l\geq l_{0}+1}|x_{l}|^{2}
$$
and $|\lambda _{l}^{n_{p}}-1|$ tends to $0$ for each $l\geq 2$,
it follows that $||D^{n_{p}}x-x||\to 0  $ as $n_{p}\to +\infty $
for any $x\in H$, hence $||T^{n_{p}}x-x||\to 0  $ which is the conclusion of Theorem \ref{th2}.
\end{proof}

\subsection{A characterization of uniformly rigid sequences for linear dynamical systems}
We now prove Theorem \ref{th2bis}. Clearly $(3)\Rightarrow (2)$. The implication $(2)\Rightarrow (1)$ is obvious: using the notation of Section 4.2 above, $||T^{n_{k}}e_{\lambda }-e_{\lambda }||$ tends to $0$ uniformly on $\sigma  _{p}(T)\cap \T=:K$ which is uncountable, i.e. $|\lambda ^{n_{k}}-1|$ tends to $0$ uniformly on $K$.
\par\smallskip
The converse implication $(1)\Rightarrow (3)$ follows from Fact \ref{fact0},
 the proof of Theorem \ref{th2ter} above and Lemma \ref{lem4} below. First replacing $K$  by a compact perfect subset of its closure, and then using Fact \ref{fact0}, we can suppose that $K$ is such that $|\lambda ^{n_{k}}-1|$ tends to $0$ uniformly on $K$ and for any $\varepsilon >0$ there exists a $\nu \in K\setminus \{1\}$ such that $\sup_{k\geq 0}|\nu ^{n_{k}}-1|\leq \varepsilon $. Then

\begin{lemma}\label{lem4}
Under the assumption above on $K$, there exists a perfect compact subset $K'$ of $\T$ such that $(K',d_{(n_{k})})$ is \sep\ and $\lambda ^{n_{k}}$ tends to $1$ uniformly on $K'$.
\end{lemma}

\begin{proof}[Proof of Lemma \ref{lem4}]
The proof runs along the same lines as that of Lemma \ref{lem5}: we
start from elements $\mu _{n}$, $n\geq 1$, of $K\setminus \{1\}$ having the same properties as in Lemma \ref{lem5}
(which we know exist - this is why we had to use Fact \ref{fact0}), and we construct the unimodular numbers $\lambda _{s}$, $s\in 2^{\omega }$, as in Lemma \ref{lem5}, with
$$|\lambda _{s}^{n_{k}}-1|\leq|\lambda_{(s_{1},\ldots, s_{p})}^{n_{k}}-1|+\frac{2}{3}4^{-p}d_{(n_{k})}(\mu _{p},
\overline{\mu}_{p})$$ for any $k\geq 0$, $p\geq 1$, and $s \in
2^{\omega }$. Let $\gamma >0$, and take $p$ such that the second term
is less than $\gamma /2$. Then for any $s\in 2^{\omega }$ and any
$k\geq 0$ we have
$$|\lambda _{s}^{n_{k}}-1|\leq |\lambda _{(s_{1},\ldots, s_{p})}^{n_{k}}-1|+\frac{\gamma }{2}\leq
|\mu _{1}^{n_{k}}-1|+\ldots+|\mu _{p}^{n_{k}}-1|+\frac{\gamma }{2}\leq p\,||\lambda ^{n_{k}}-1||_{\infty ,K}+\frac{\gamma }{2}\cdot$$
Take $\kappa $ such that for any $k\geq \kappa $, $||\lambda ^{n_{k}}-1||_{\infty ,K}\leq \gamma /(2p)$: we have $|\lambda _{s}^{n_{k}}-1|\leq \gamma $ for any $s\in 2^{\omega }$ and $k\geq \kappa $, and this shows that $\lambda ^{n_{k}}$ tends to $1$ uniformly on the set $K'=\{\lambda _{s} \textrm{ ; } s\in 2^{\omega }\}$. Since $(K',d_{(n_{k})})$ is \sep, Lemma \ref{lem4} is proved.
\end{proof}

Now in the construction of the \op\ $T$, we choose the coefficients $\lambda _{l}$ in the set $K'$ given by Lemma \ref{lem4}. We have seen in the proof of Theorem \ref{th2} above that $||T^{n_{k}}-D^{n_{k}}||$ tends to $0$ as $n_{k}$ tends to infinity. So it suffices to prove that with the additional uniformity assumption of Theorem \ref{th2bis}, $||D^{n_{k}}-I||=\sup_{l\geq 1}|\lambda _{l}^{n_{k}}-1|$ tends to $0$ as $n_{k}$ tends to infinity. But $||D^{n_{k}}-I||\leq ||\lambda ^{n_{k}}-1||_{\infty ,K'}$ which tends to $0$, so  our claim is proved.
\par\medskip
\begin{proof}[Proof of Corollary \ref{cor}]
 If $(n_{k})_{k\geq 0}$ is any sequence with $n_{k+1}/n_{k}\to +\infty $, or if $n_{k}| n_{k+1}$ for any $k\geq 0$ and $\limsup n_{k+1}/n_{k}=+\infty $, we have seen in Propositions \ref{prop3} and \ref{prop22}
 that Theorem \ref{th2bis} applies, proving Corollary \ref{cor}. Theorem \ref{th2bis} also applies to the sequences $(q_{n})_{n\geq 0}$ considered in Example \ref{ex2}.
We thus obtain examples, in the linear framework, of \mpt s on a Hilbert space which are both \wmx\ in the \mea-theoretic sense and \ur.
\end{proof}

\begin{remark}
If $(n_k)_{k\geq 0}$ is such that $n_{k}| n_{k+1}$ for any $k\geq 0$ and $\limsup n_{k+1}/n_{k}=+\infty $, the proof of Proposition \ref{prop22} shows that the set $K=\{\lambda_{\varepsilon} \textrm{ ; }\varepsilon\in\{0,1\}^{\N}\}$ contains a dense subset of numbers $\lambda$ which are $N^{th}$ roots of $1$ for some $N\geq 1$. Hence, in all the constructions of operators $T=D+B$ considered here, it is possible to choose the numbers $\lambda_l$, $l\geq 1$, as being $N^{th}$ roots of $1$. In this way the \op\ $T$ becomes additionally chaotic (i.e. it is topologically transitive and has a dense set of periodic vectors). This gives further examples of chaotic \ops\ which are not topologicaly mixing (the first examples of such \ops\ were given in \cite{BaG1}), and shows in particular that there exist chaotic \ops\ which are uniformly rigid.
\end{remark}

\par\bigskip

\textbf{Acknowledgements.} We are deeply grateful to Jean-Pierre Kahane for his kind help with the proof of Proposition \ref{prop33}. We are also grateful to Mariusz Lema\'nczyk, Herv\'e Queff\'elec, Martine Queff\'elec and Maria Roginskaya for
helpful discussions.

\end{document}